\documentclass[reqno]{amsart}
\usepackage[utf8]{inputenc}
\usepackage{amsmath,amsfonts,amsthm,amssymb} 
\usepackage{mathtools} 
\usepackage[margin=1in]{geometry} 
\usepackage[shortlabels]{enumitem}
\usepackage{aliascnt}
\usepackage{asymptote,graphicx}
\usepackage{todonotes}
\usepackage[font=small,margin=0.5in]{caption} 
\usepackage{mwe} 
\usepackage{hyperref}
\usepackage[nameinlink]{cleveref} 
\usepackage[all]{xy}

\graphicspath{{diagrams/}} 

\theoremstyle{plain}
\newtheorem{theorem}{Theorem}[section] 
\newtheorem{prop}[theorem]{Proposition} 
\newtheorem{lemma}[theorem]{Lemma}
\newtheorem{corollary}[theorem]{Corollary} 

\crefname{lemma}{lemma}{lemmas}
\crefname{prop}{proposition}{propositions}
\Crefname{corollary}{corollary}{corollaries}

\theoremstyle{definition}
\newtheorem{claim}{Claim}
\crefname{claim}{claim}{claims}

\theoremstyle{definition}
\newtheorem{definition}[theorem]{Definition}
 
\crefname{definition}{definition}{definitions}

\theoremstyle{remark}
\newaliascnt{remark}{theorem} 
\newtheorem{remark}[remark]{Remark} 
\aliascntresetthe{remark}

\crefname{section}{section}{sections}
 
\crefname{subsection}{section}{sections}
\crefname{equation}{equation}{equations}

\newcommand\NN{\mathbb N}

\DeclareMathOperator{\id}{id}
\DeclareMathOperator{\im}{im}
\newcommand\TT{\mathcal T}
 
\newcommand\norm[1]{\left\lVert#1\right\rVert}

\title{Classification of tight contact structures on a solid torus}
\author{Zhenkun Li and Jessica Zhang}
\date{}

\begin{document}

\bibliographystyle{plain}

\maketitle

\begin{abstract}
It is a basic question in contact geometry to classify all non-isotopic tight contact structures on a given $3$-manifold. If the manifold has a boundary, we need also specify the dividing set on the boundary. In this paper, we answer the classification question completely for the case of a solid torus by writing down a closed formula for the number of non-isotopic tight contact structures with any given dividing set on the boundary of the solid torus. Previously, only a few special cases were known due to work by Honda. 
\end{abstract}

\section{Introduction}
In the past few decades, $3$-dimensional contact geometry has been studied extensively and has become one of the most important and powerful tools in the field of $3$-dimensional topology. It is a fundamental problem to classify all non-isotopic tight contact structures on a given $3$-manifold, with or without boundary. In contrast to the fast development of the field, the classification problem is still widely open. For closed $3$-manifolds, only some simple manifolds have been understood, which we summarize as follows:
\begin{itemize}
\item $S^3$ and $S^1\times S^2$: See Eliashberg \cite{eliashberg1992contact}.
\item $T^3$: See Kanda \cite{kanda1997classification}.
\item Lens spaces: See Etnyre \cite{etnyre2000lens}, Giroux \cite{giroux2000lens}, and Honda \cite{honda2000classification}.
\item Torus bundles over circles and circle bundles over closed surfaces: See Giroux \cite{giroux2000lens,giroux2001fibre} and Honda \cite{honda2000classification2}.
\item The Poincar\'e sphere: See Honda and Etnyre \cite{honda2001nonexistence}.
\item Small Seifert fibred spaces: This family of $3$-manifolds was first studied by Ghiggini and Sch\"onen\-berger \cite{ghiggini2003sfs}. Then Wu \cite{wu2006smallsfs} was able to classify tight contact structures on all but three families of small Seifert fibred spaces. Later those remaining families were studied by Ghiggini, Lisca, and Stipsicz \cite{Ghiggini2006sfs} and B\"ulent \cite{bulent2020sfs}, and the relation between the contact structures and Heegaard Floer homology on such manifolds was studied by Ghiggini, Lisca, and Stipsicz \cite{ghiggini2007sfs} and Matkovi\v{c} \cite{matkovic2018classification}, among others.
\end{itemize}

For a $3$-manifold $M$ with boundary, we impose the additional condition that the boundary $\partial M$ is convex. According to Giroux \cite{giroux1991convexite}, the local behavior of the contact structure near $\partial M$ is determined by a set of distinguished curves $\Gamma$ on $\partial M$, which we call the dividing set. As such, to classify tight contact structures for such an $M$, we need to classify tight contact structures for the pair $(M,\Gamma)$ for all possible pairs. The dividing set satisfies some constraints, such as Giroux's criterion in \Cref{thm: giroux's criterion}.

Very little is known for the classification problem on $3$-manifolds with boundaries. Some special cases of the following list of $3$-manifolds with boundaries have been studied: 
\begin{itemize}
\item $[0,1]\times T^2$: See Honda \cite{honda2000classification}.
\item $S^1\times F$ for a compact oriented surface with boundary: See Makar-Limanov \cite{maker-limanov1994solidtori} and Honda \cite{honda2000classification} for the case $F=D^2$, and Honda \cite{honda2000classification2} for a more general $F$.
\item $[0,1]\times \Sigma$ for a closed oriented surface $\Sigma$ with genus at least $2$: See Honda, Kazez, and Mati\'c \cite{honda2003hyperbolic} and Cofer \cite{cofer2004contact}.
\item The genus-two handlebody: See Ortiz \cite{ortiz2015convex}.
\end{itemize}
To our knowledge, apart from these partial results, the only (irreducible) $3$-manifold with boundary on which the tight contact structures have been fully classified is the case of a $3$-ball $D^3$: Eliashberg \cite{eliashberg1992contact} showed that there is a unique possible dividing set on $\partial D^3$ and a unique tight contact structure on this dividing set. In this paper, we add a second manifold to this very short list: the solid torus $S^1\times D^2.$ 

For a solid torus $M=S^1\times D^2$, a possible dividing set $\Gamma$ on its boundary $\partial M$ can be parametrized by a triple $(n,p,q)$, where
\begin{itemize}
\item $2n$ is the number of components of $\Gamma$, 
\item $p$ is the number of times each component of $\Gamma$ goes around the longitude $S^1\times\{1\}$, and
\item $q$ is the number of times each component of $\Gamma$ goes around the meridian $\{1\}\times \partial D^2$. 
\end{itemize}
On a solid torus, we can always perform Dehn twists along meridian disks to change the pair $(p,q)$ into $(p,q+kp)$. We also have that $(n,p,q)$ and $(n,-p,-q)$ parametrize the same dividing set. Following the convention laid out by Honda \cite{honda2000classification}, we use parametrizations of the form $(n,-p,q)$ and always assume that $0<q\leq p$ and $\gcd(p,q)=1$. Note that by Honda \cite{honda2000classification}, there is no tight contact structures if $p=0$. To better present our main result, we also adopt the following notation from Honda \cite{honda2000classification}: When $(p,q)\neq(1,1)$, write \[-\frac pq=[r_0,r_1,\dots,r_k]=r_0-\frac1{r_1-\frac1{r_2-\dots\frac1{r_k}}},\] where $r_i\le-2$ are integers, and define \begin{align*}r&=|(r_0+1)(r_1+1)\dots(r_{k-1}+1)r_k|\\ s&=|(r_0+1)(r_1+1)\dots(r_{k-1}+1)(r_k+1)|.\end{align*}
When $(p,q)=(1,1)$, define $r=1$ and $s=1$. The main theorem of the paper is the following. 

\begin{theorem}\label{thm: main theorem}
Suppose $M=S^1\times D^2$ is a solid torus. Let $\Gamma$ be a dividing set on $\partial M$ parametrized by $(n,-p,q)$, where $0<q\leq p$ and $\gcd(p,q)=1$. Let the pair of integers $(r,s)$ be defined as above.  Then the number of isotopy classes of tight contact structures on $M$ with dividing set $\Gamma=(n,-p,q)$ is precisely
\[N(n,-p,q)=C_n((r-s)n+s),\]
where $C_n$ is the $n$-th Catalan number:
\[C_n=\frac{1}{n+1}\binom{2n}n.\]
\end{theorem}

The cases when $n=1$ or $p=q=1$ have already been studied by Honda \cite{honda2000classification,honda2000classification2}. In \cite{honda2002gluing}, Honda also introduced an algorithm that theoretically classifies all possible tight contact structures on any handlebodies. The algorithm specific to the case of a solid torus was later simplified by Cofer \cite{cofer2016classification}. In this paper, we adopt another idea to compute the number $N(n,-p,q)$ in \Cref{thm: main theorem}. The key tool is that of bypasses, introduced by Honda \cite{honda2000classification}. 

A bypass is a half-disk, carrying a special contact structure, attached to a convex surface along an arc that intersects the dividing curve on the surface exactly three times. The half-disk does not change the topology of the surface but changes the local contact structure instead. If a bypass is attached from the interior side to the boundary of a $3$-manifold $M$, equipped with a contact structure $\xi$, we can think of it as peeling off a collar of $\partial M$ to obtain a new $3$-manifold $M'$ together with a new contact structure $\xi'$. For the case of a solid torus $M=S^1\times D^2$, usually the dividing set of $\xi'$ is simpler than that of $\xi$. Also, the bypass has a duality: If $\xi'$ is obtained from $\xi$ by peeling off a collar of $\partial M$, as above, then $\xi$ is obtained from $\xi'$ by gluing back the collar, which can be realized as attaching a bypass from the exterior of $\partial M'$. Studying this duality more carefully and using Floer theory, we are able to show that the correspondence $\xi\leftrightarrow \xi'$ always gives rise to a bijection in a proper sense on a solid torus. Hence we can apply induction and obtain a recurrence relation for the sequence $N(n,-p,q)$. Finally, we use this recurrence and previously known cases to derive a closed formula for $N(n,-p,q)$. 

{\bf Organization of the paper.} We begin in \Cref{sec: preliminaries} with several preliminaries, including those of convex surfaces, dividing sets, and bypasses. Then in \Cref{sec: solid torus}, we focus specifically on the solid torus $S^1\times D^2$ and prove several lemmas involving possible bypasses on it. Finally, we prove the main theorem in \Cref{sec: proof}. 

{\bf Acknowledgement.} This collaboration was made possible by the PRIMES-USA program. The first author was supported by his advisor Tom Mrowka’s NSF Grant 1808794. The authors would like to thank John Etnyre and Ko Honda for explaining some of the background of the problem.

\section{Preliminaries} \label{sec: preliminaries}
\subsection{Contact structures}
In general, contact structures can be defined on manifolds of dimension $2n+1$. We will, however, restrict our attention to (oriented) 3-manifolds. 

\begin{definition}
Let $M$ be a compact 3-manifold. If there exists a 1-form $\alpha$ on $M$ such that $\alpha\wedge d\alpha>0$ everywhere on $M$, then we call $\xi=\ker\alpha$ a \emph{(positive) contact structure} on $M$ and we say that $\alpha$ is a \emph{contact form} for $\xi$. Equivalently, a contact structure on $M$ is a maximally nonintegrable 2-plane distribution. 
\end{definition} 

It is known that every oriented compact smooth 3-manifold admits a contact structure \cite{martinet1971formes}. Note that, in general, when we refer to contact structures, we are in fact referring to the isotopy types of the contact structures; in particular, we henceforth consider contact structures to be unique only up to isotopy. 

\begin{definition}
A curve $L\subset M$ is called \emph{Legendrian} if at every point $x\in L$, we have $T_xL\subset\xi_x$. We distinguish between \emph{Legendrian curves}, which are closed, and \emph{Legendrian arcs}, which are not. 
\end{definition}

\begin{definition}
The \emph{twisting number} of the Legendrian curve $L$ relative to a framing $Fr$ of the curve $L$ is denoted $t(L,Fr)$ and is equal to the integer number of counterclockwise twists of $\xi$ along $L$ relative to $Fr$. 
\end{definition}

Often, when there is a clear surface $\Sigma$ containing $L$, we will use $t(L)$ as shorthand to denote $t(L,Fr_\Sigma)$. 

\begin{definition}
A 3-dimensional contact manifold $(M,\xi)$ is \emph{overtwisted} if there exists a disk $D^2\subset M$ whose boundary is Legendrian and its twisting number $t(\partial D^2)=0$. Such a disk is called an \emph{overtwisted disk}. A contact structure is called \emph{tight} if it does not contain an overtwisted disk. 
\end{definition} 

Work by Eliashberg \cite{eliashberg1989class} gives a classification of overtwisted contact structures on 3-manifolds. As such, we focus in this paper on classifying tight contact structures instead. 

\subsection{Convex surfaces} 
Our goal is to classify tight contact structures based on their dividing sets, which are defined on a specific class of surfaces known as convex surfaces, first introduced by Giroux \cite{giroux1991convexite}. 

\begin{definition}
If $\Sigma$ is embedded in $M$, then we call $\Sigma$ a \emph{convex surface} if there exists a contact vector field $X$, i.e., a vector field $X\subset TM$ whose flow preserves $\xi$, which is everywhere transverse to the surface $\Sigma$. 
\end{definition} 

\begin{definition}
Let $\Sigma\subset M$ be a convex surface and $X$ be a contact vector field that is everywhere transverse to $\Sigma$. Then the \emph{dividing set} on $\Sigma$ is the collection of points \[\Gamma_\Sigma=\{x\in\Sigma:X_x\in\xi_x\}.\] In other words, it is the set of points in $\Sigma$ such that the corresponding vector in $X$ belongs to the contact structure. We call a curve $\gamma\subset\Sigma$ a \emph{dividing curve} if it is a connected component of $\Gamma_\Sigma$. 
\end{definition} 

In general, the isotopy type of the dividing set is independent of the vector field $X$, so we can consider $\Gamma_\Sigma$ to be ``the'' dividing set of $\Sigma$. We denote by $\#\Gamma_\Sigma$ the number of components of the dividing set. 

\begin{theorem}[Giroux's criterion \cite{giroux1991convexite}]\label{thm: giroux's criterion}
Let $(M,\xi)$ be a contact manifold and let $\Sigma$ be a convex surface. If $\Sigma\ne S^2$, then $\Sigma$ has a tight neighborhood if and only if $\Gamma_\Sigma$ has no homotopically trivial closed components. If $\Sigma=S^2$, then $\Sigma$ has a tight neighborhood if and only if $\#\Gamma_\Sigma=1$.  
\end{theorem} 

Note that convex surfaces are particularly useful because Giroux's criterion gives us a way to determine whether a given convex surface has a tight neighborhood or not solely based on its dividing set $\Gamma_\Sigma$. Moreover, there is a particularly nice formula for the twisting number of a Legendrian curve on a convex surface.

\begin{theorem}[Kanda \cite{kanda1998thurston}] \label{thm: twisting number relative to framing} 
If $L$ is a Legendrian curve on a convex surface $\Sigma$, then the twisting number of $L$ relative to the framing induced by $\Sigma$ is \[t(L,Fr_{\Sigma})=-\frac12\#(L\cap\Gamma_\Sigma).\] 
\end{theorem}

\subsection{Legendrian and convex realization} 
The most useful surfaces for us will be (compact) convex surfaces with Legendrian boundary. Although Legendrian curves and convex surfaces appear to be fairly ''special,'' it turns out that both are in some sense ''generic.'' 

Indeed, most curves can be realized as Legendrian. In particular, the Legendrian realization principle, first proved by Kanda \cite{kanda1997classification} and later strengthened by Honda \cite{honda2000classification}, allows us to realize almost all embedded curves as Legendrian ones. 

\begin{definition}
Suppose $\mathcal C\subset\Sigma$ is a disjoint union of closed curves and arcs such that $\mathcal C$ is transverse to $\Gamma_\Sigma$, every arc of $\mathcal C$ begins and ends on $\Gamma_\Sigma$, and every component of $\Sigma\setminus\mathcal C$ intersects $\Gamma_\Sigma$ nontrivially. Then we say that $\mathcal C$ is \emph{nonisolating}. 
\end{definition} 

\begin{theorem}[Legendrian realization]\label{thm: legendrian realization}
Any nonisolating collection $\mathcal C$ of closed curves and arcs can be realized as Legendrian in the sense that there is an isotopy $\phi_s$ with $s\in[0,1]$ such that the following hold: 

\begin{enumerate}[(1)] 
\item $\phi_0=\id_\Sigma$, 
\item $\phi_s(\Sigma)$ is convex for all $s$, 
\item $\phi_s\rvert_{\partial\Sigma}=\id_{\partial\Sigma}$ for all $s$, 
\item $\phi_1(\Gamma_\Sigma)=\Gamma_{\phi_1(\Sigma)}$, and 
\item $\phi_1(\mathcal C)$ is Legendrian. 
\end{enumerate}
\end{theorem}

Note that every closed curve $C\subset\Sigma$ which is transverse to $\Gamma_\Sigma$ and which intersects $\Gamma_\Sigma$ nontrivially is nonisolating, and is thus realizable as Legendrian. 

There is a similar principle which allows us to perturb a given surface to be convex. 

\begin{prop}[Giroux \cite{giroux1991convexite}] 
Suppose that $\Sigma$ is closed, oriented, and properly embedded. Then there is a $C^\infty$-small isotopy which makes $\Sigma$ convex. 
\end{prop} 

This proposition allows us to always tacitly assume that our three-manifold $M$ has a convex boundary, which we will henceforth do. Honda \cite{honda2000classification} proved a version of this proposition for surfaces with boundary. 

\begin{prop} \label{thm: convex flexibility} 
Suppose that $\Sigma$ is compact, oriented, and properly embedded with Legendrian boundary. Suppose that $t(\gamma,Fr_\Sigma)\le0$ for all connected components $\gamma$ of $\partial\Sigma$. Then there is a $C^0$-small perturbation near the boundary which fixes $\partial\Sigma$, followed by a $C^\infty$-small perturbation of the peturbed surface which fixes a neighborhood of $\partial\Sigma$, such that the final surface is convex. 
\end{prop} 

With this, along with Legendrian realization, it will be possible for us to perturb surfaces to be convex with Legendrian boundary. 

\subsection{Bypasses} 
The fundamental tool for us will be that of the bypass attachment, either from the interior or from the exterior, which will allow us to simplify the contact structure. 

\begin{definition} \label{defn: bypass}
Let $\Sigma$ be a convex surface on the contact 3-manifold $(M,\xi)$. Then a \emph{bypass} is a convex half-disk $B$ with Legendrian boundary such that the following requirements are satisfied: 
\begin{enumerate}[(1)]
\item The arc $\alpha=B\cap\Sigma$ is Legendrian and intersects $\Gamma_\Sigma$ at exactly three points $p_1,p_2,p_3$, where $p_1$ and $p_3$ are the endpoints of $\alpha$, 
\item The half-disk $B$ is transverse to the surface $\Sigma$, and 
\item The twisting number of the boundary of $B$ is $t(\partial B)=-1$. 
\end{enumerate} 
\end{definition} 

An example of a bypass can be seen in \Cref{fig:bypass-disk}. 

\begin{figure}[htbp]
\begin{center}
\includegraphics[scale=0.7]{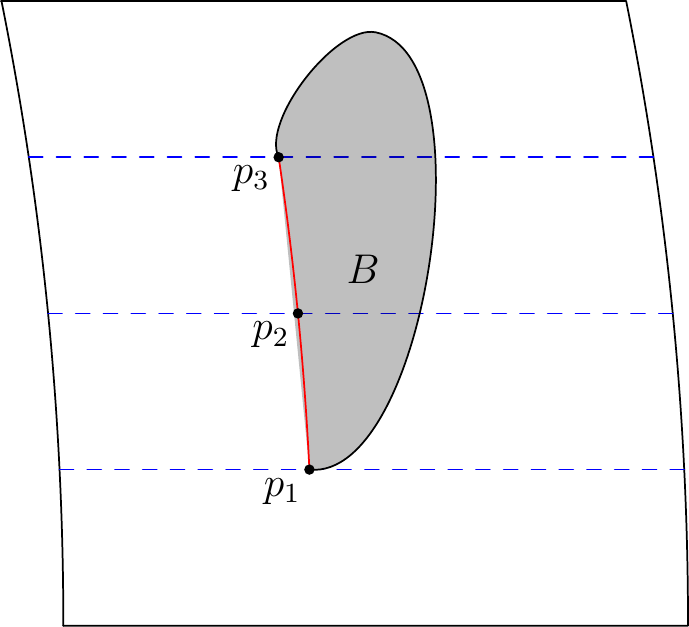} 
\caption{A bypass disk $B$ with attaching arc $\alpha$ shown in red.}
\label{fig:bypass-disk}
\end{center}
\end{figure}

With a bypass as in the above definition, the arc $\alpha$ is known as the \emph{arc of attachment}, and we say that $B$ is a bypass along $\alpha$ on $\Sigma$. We call a bypass \emph{exterior} or \emph{interior} depending on whether $B$ is attached on the side coinciding with the orientation of $\Sigma$ or not. 

Consider the change in the dividing set shown in \Cref{fig:bypass-ds}. In words, it is performed by splitting each of the three arcs in the dividing set that intersects the attaching arc at the point of intersection. Then, for the two points on either end of the attaching arc, if we consider the attaching arc to be going up from the identified endpoint, then we join the two arcs on the left. The two remaining arcs (which would be top left and bottom right from either direction) are then joined. 

\begin{figure}[htbp] 
\begin{center} 
\includegraphics[scale=0.8]{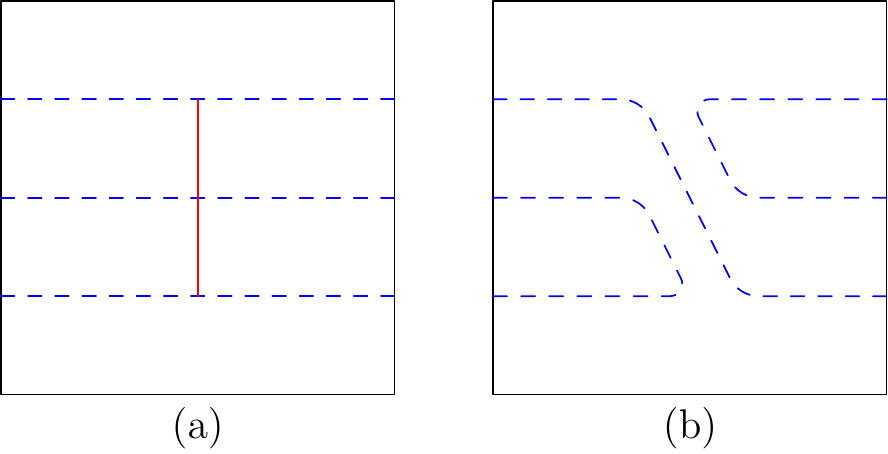} 
\caption{Attaching a bypass on the exterior along the red Legendrian arc $\alpha$ in (a) results in a change in dividing set as seen in (b).}
\label{fig:bypass-ds}
\end{center} 
\end{figure} 

\begin{lemma}[Bypass Attachment Lemma \cite{honda2000classification}] 
If $B$ is an exterior bypass along a convex surface $\Sigma$, then there exists a neighborhood $N$ of $\Sigma\cup D$ with convex $\partial N=\Sigma-\Sigma'$ such that the new dividing set $\Gamma_{\Sigma'}$ is obtained from $\Gamma_\Sigma$ in the manner detailed above. 
\end{lemma} 

Note that, in the case of an interior bypass, we can simply imagine reversing the orientation of $\Sigma$. Then we have an exterior bypass instead and can apply the bypass attachment lemma. This is equivalent to performing the mirror operation which involves cutting the components at the attaching arc, and  joining the bottom right edges and the top left edges. 

Even if we do not know of the existence of a bypass, that is, a half-disk satisfying the properties of a bypass half-disk, the bypass attachment lemma tells us what would happen to the dividing set if such a bypass were to exist. Thus we at times consider \emph{abstract bypass attachments} in which we consider how the dividing set would be affected if a bypass were to exist. 

In the case that an abstract bypass move does not alter the dividing set (up to isotopy), we call it \emph{trivial}. The following proposition shows that such a bypass is indeed trivial in the sense that it does not alter the contact structure. 

\begin{prop}[Honda \cite{honda2002gluing}] \label{prop: trivial bypasses are trivial} 
Let $\Sigma$ be a closed or compact convex surface with Legendrian boundary. Suppose that a trivial bypass $B$ is attached to $\Sigma$ along the attaching arc $\delta\subset\Sigma$. Then there exists a neighborhood of $\Sigma\cup_\delta B$ which is isotopic to the standard $I$-invariant neighborhood of $\Sigma$, which is simply $\Sigma\times[0,1]$ such that $\Sigma=\Sigma\times\{0\}$. 
\end{prop}

In general, even though we do not get bypasses ``for free,'' bypasses are relatively abundant on convex surfaces in contact manifolds. 

\begin{definition}
If $\Sigma$ is a convex surface with nonempty Legendrian boundary, then a component $\gamma\subseteq\Gamma_\Sigma$ is called a \emph{boundary-parallel} dividing curve if it cuts off a half-disk $B$ in $\Sigma$ such that $B\cap\Gamma_\Sigma=\gamma$. 
\end{definition} 

\begin{prop}[Honda \cite{honda2000classification}] \label{prop: bypass abundance} 
Suppose that $\Sigma$ is a convex surface with Legendrian boundary and that $\gamma$ is boundary-parallel. If $\Sigma$ is not a disk with twisting number $t(\partial\Sigma,Fr_\Sigma)=-1$, then there exists a bypass half-disk containing the half-disk cut off by $\gamma$. 
\end{prop}

This is illustrated in \Cref{fig:boundary-parallel}. For simplicity, we at times call a bypass induced by a boundary-parallel curve a \emph{boundary-parallel} bypass. Similarly, we say that the associated half-disk is boundary-parallel. 

\begin{figure}[htbp] 
\centering
\includegraphics[scale=0.8]{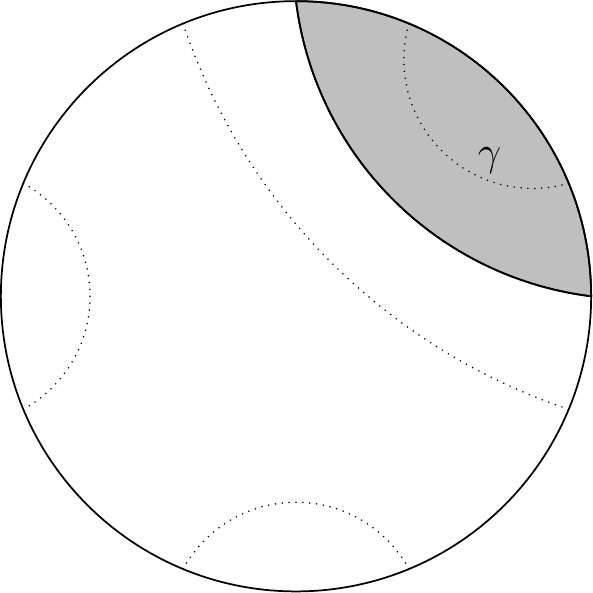}
\caption{The dividing set on this convex surface is in dotted lines and the shaded area is a boundary-parallel bypass half-disk. Note that this particular state admits three distinct (though possibly isotopic) half-disks induced by boundary-parallel curves.}
\label{fig:boundary-parallel}
\end{figure} 

Finally, we will also need the following two results due to Honda. 

\begin{prop}[Honda \cite{honda2000classification}] \label{prop: reverse bypass} 
If $(M,\xi)$ is a contact $3$-manifold admitting an exterior (respectively, interior) bypass, the attachment of which takes $\xi$ to a new contact structure $\xi'$, then we can turn the bypass upside down to obtain an interior (respectively, exterior) bypass that takes $(M,\xi')$ to $(M,\xi)$. 
\end{prop}

\begin{prop}[Honda \cite{honda2000classification}] \label{prop: bypasses commute}
Bypass attachment is commutative in the sense that if we have two disjoint attaching arcs $\alpha$ and $\beta$, then the contact structure that results from attaching a bypass at $\beta$ and then $\alpha$ is the same as the one resulting from attaching $\alpha$ and then $\beta$. 
\end{prop}

\subsection{Known results for a solid torus}\label{sec: known results}
In this section, we summarize a few known classification results for the solid torus. These will prove useful to us in our proof of the main theorem. 

Recall that $N(n,-p,q)$ denotes the number of tight contact structures with dividing set $\Gamma_{\partial M}=(n,-p,q)$. 

\begin{theorem}[Honda \cite{honda2000classification2}] \label{thm: n10 case}
For every $n\in\NN$, there are precisely $C_n$ tight contact structures whose dividing set can be parametrized as $(n,-1,1)$, where $C_n$ is the $n$-th Catalan number. 
\end{theorem} 

For any relatively prime integers $p$ and $q$ with $0<q\le p$, recall that we can write $-\frac pq$ as \[-\frac pq=[r_0,r_1,\dots,r_k]=r_0-\frac1{r_1-\frac1{r_2-\dots\frac1{r_k}}},\] where $r_i\le-2$ is an integer for each $i$. The only exception is for $p=q=1$, in which case we write $-\frac pq$ as simply $[r_0]=[-1]$. 

\begin{theorem}[Honda \cite{honda2000classification}] \label{thm: 1pq case}
Consider a dividing set on $\partial M$ parametrized by $(1,-p,q)$. Then \[N(1,-p,q)=|(r_0+1)(r_1+1)\dots(r_{k-1}+1)r_k|,\] where the $r_i$ are the coefficients of the continued fraction expansion of $-\frac pq=[r_0,r_1,\dots,r_k]$. 
\end{theorem} 

Besides these two partial classification results for the solid torus, we will also make use of the following propositions, which tell us how attaching a single bypass will affect the dividing set on $\partial M=T^2$. 

\begin{prop}[Honda \cite{honda2002gluing}] \label{prop: bypass attachment possibilities}
If $\Sigma=T^2$, then attaching a valid bypass to $\Sigma$ can only affect the dividing set in one of the following ways: 
\begin{enumerate}[(1)]
\item The bypass attachment is trivial, and so the dividing set remains the same,
\item The number of components in $\Gamma_\Sigma$ decreases by 2 (provided $\#\Gamma_\Sigma=2n>2)$, 
\item The number of components in $\Gamma_\Sigma$ increases by 2, or 
\item The new dividing curve is achieved from $\Gamma_\Sigma$ via a positive Dehn twist. 
\end{enumerate} 
\end{prop} 

\begin{remark} 
In fact, from the proof of this proposition, we can achieve a slightly stronger statement which tells us which of the four cases we get based on whether the bypass's attaching arc intersects three distinct components or not and, if it does not, which components are the same. In particular, suppose the attaching arc of the bypass intersects $\Gamma_\Sigma$ at the points $P_i\in\gamma_i$, in that order, where $\gamma_i\subseteq\Gamma_\Sigma$ are dividing curves for $i=1,2,3$. Then the only relevant cases for us are (1) if $\gamma_i$ are all distinct, in which case the only possibility is Case (2) above, and (2) if $\gamma_1=\gamma_3\ne\gamma_2$, which will automatically result in Case (4). 
\end{remark}


\subsection{Input from Floer theory}\label{subsec: floer theory}
In this section, we summarize known results in Floer theory that will help us understand bypasses on a solid torus. It is worth mentioning that the different branches of Floer theory work equally well; we simply choose to present the discussion below using instanton Floer theory. Also, Floer theory and bypass triangles, which will be introduced in the current subsection, work more generally for arbitrary $3$-manifolds $M$ with a boundary and a special set of curves known as the sutures on the boundary. However, for consistency, we focus on the case of a solid torus, and the dividing set parametrized by $(1,-p,q)$. We will continue using the term "dividing set" instead of "suture".

Suppose $M$ is a solid torus and $\Gamma=(n,-p,q)$ is a dividing set on $\partial M$. Instanton Floer theory associates to each pair $(M,\Gamma)$ a finite-dimensional vector space over $\mathbb{Z}_2$. 
\begin{theorem}[Kronheimer and Mrowka \cite{kronheimer2010knots}]
	For any given pair $(M,\Gamma)$, there is a well defined finite-dimensional vector space over $\mathbb{C}$ which we denote by $SHI(M,\Gamma)$, associated to $(M,\Gamma)$.
\end{theorem}

\begin{remark}
It bears mentioning that Kronheimer and Mrowka only proved that $SHI(M,\Gamma)$ is well-defined up to isomorphism. Then in \cite{baldwin2015naturality}, Baldwin and Sivek  proved that $SHI(M,\Gamma)$ is well defined up to multiplication by an element in $\mathbb{C}^*$. In this paper, we overlook this remaining ambiguity, since we only care about whether an element in $SHI(M,\Gamma)$ is zero or not, and whether given elements are linearly dependent or not.
\end{remark}

\begin{theorem}[Baldwin and Sivek \cite{baldwin2016instanton}]\label{thm: contact elements}
	Suppose $\xi$ is a contact structure on $(M,\Gamma)$. Then there is an element $\phi(\xi)\in SHI(-M,-\Gamma)$ associated to the contact structure $\xi$. Furthermore, if $\xi$ is overtwisted, then $\phi(\xi)=0$. 
\end{theorem}

\begin{definition}
We call the element $\phi(\xi)$ in \Cref{thm: contact elements} a \emph{contact element} associated to the contact structure $\xi$.
\end{definition}

In \cite{li2019direct}, the first author studied the instanton Floer homology for the pair $(M,\Gamma)$, where $M$ is a solid torus and $\Gamma=(1,-p,q)$. Recall that we have assumed that $0<q\leq p$ and $(p,q)=1$. Then we have the following result.

\begin{theorem}[Li \cite{li2019direct}]\label{thm: SHI for sutured solid torus}
	Suppose $M$ is a solid torus and $\Gamma=(1,-p,q)$. Then there exists a $\mathbb{Z}$-grading on $SHI(-M,-\Gamma)$, which we write $SHI(-M,-\Gamma,i)$, so that
	\begin{equation*}
	SHI(-M,-\Gamma,i)=\left\{
	\begin{array}{cc}
		\mathbb{C}&1\leq i\leq p\\
		0&{\rm else}
	\end{array}
	\right.
	\end{equation*}
Furthermore, any tight contact structure on $(M,\Gamma)$ has a nonzero contact element, and any two tight contact structures have contact elements supported in different gradings of $SHI(-M,-\Gamma)$.
\end{theorem}

We have seen that bypasses alter the contact structures as well as the dividing sets. Their relation with Floer theory was first discussed by Honda \cite{honda2000bypass} in the context of Heegaard Floer theory, and was then extended to instanton theory by Baldwin and Sivek \cite{baldwin2018khovanov}.
\begin{theorem}[Baldwin and Sivek \cite{baldwin2018khovanov}]\label{thm: bypass induces maps on SHI}
Suppose $M'$ is a solid torus, $\Gamma'=(1,-p,q)$ is a dividing set on $\partial M$, and $\xi'$ is a contact structure on $(M,\Gamma')$. Suppose a bypass $\beta$ attached from the exterior of $\partial M$ changes the contact structure from $\xi'$ to $\xi$ and changes dividing set from $\Gamma'$ to $\Gamma$ (the topology of the $3$-manifold $M$ remaining unchanged). Then there is a map
$$\Psi_{\beta}:SHI(-M,-\Gamma')\rightarrow SHI(-M,-\Gamma)$$
so that
$$\Psi_{\beta}(\phi(\xi'))=\phi(\xi).$$
\end{theorem}

One of the most important features of a bypass is that it fits into a bypass triangle as follows: Suppose $M$ is a solid torus and $\Gamma_1=(1,-p,q)$ is a dividing set on $\partial M$. Suppose we attach a bypass along an arc $\beta$ as shown in \Cref{fig:exact-triangle}, and the dividing set is changed to $\Gamma_2$. Suppose we further attach a second bypass $\theta$, which changes the dividing set from $\Gamma_2$ to $\Gamma_3$, followed by a third bypass $\eta$, which changes the diving set from $\Gamma_3$ back to $\Gamma_1$. This bypass triangle induces an exact triangle in instanton Floer theory, a notion which is made rigorous in the following theorem. 

\begin{figure}[htbp]
\centering
\includegraphics[width=0.5\textwidth]{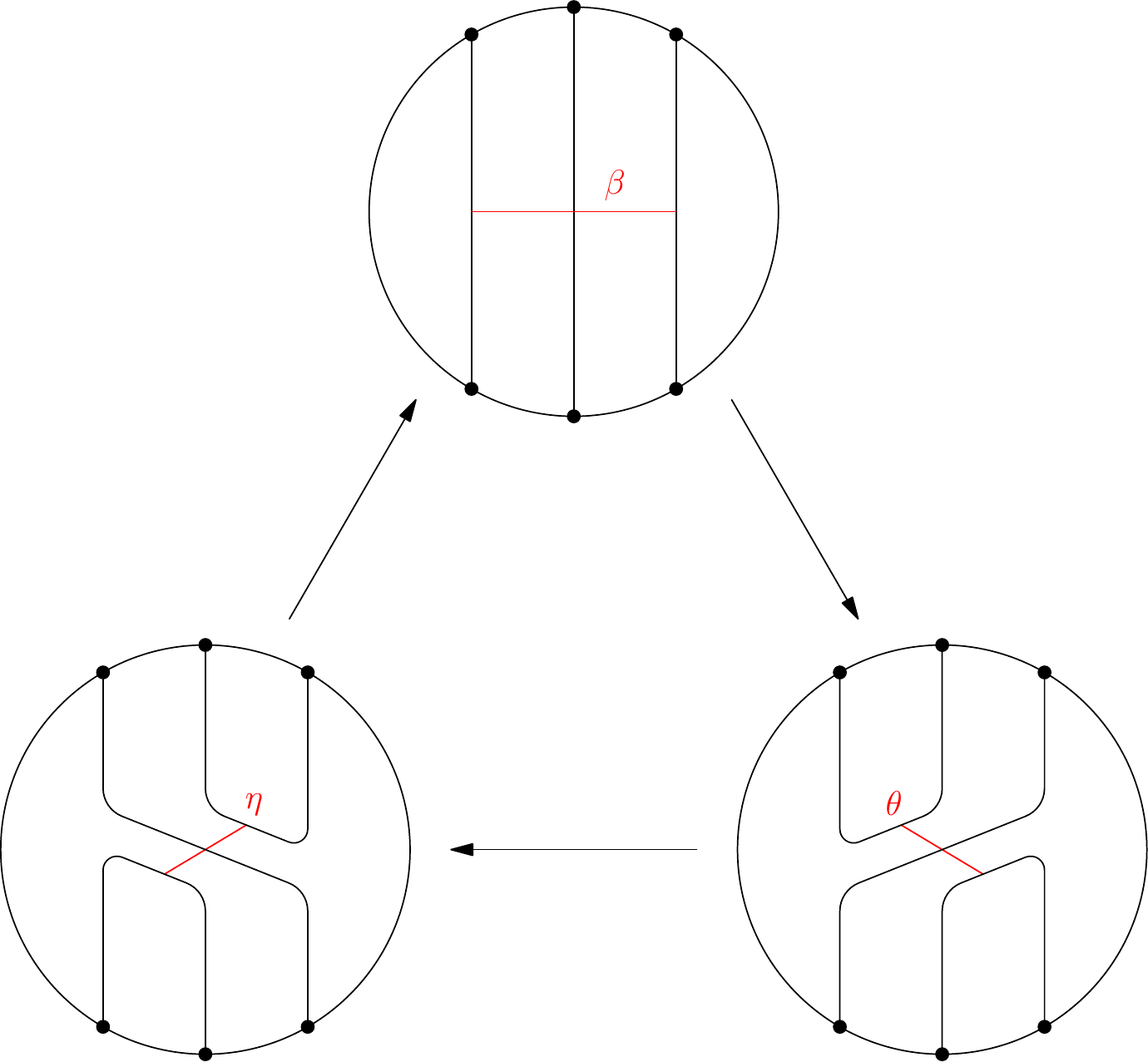}
\caption{The above bypass triangle is obtained by successively attaching exterior bypasses $\beta$, $\theta$, and $\eta$, and is exact.}
\label{fig:exact-triangle}
\end{figure}

\begin{theorem}[Baldwin and Sivek \cite{baldwin2018khovanov}]\label{thm: bypass exact triangle}
The following triangle is exact: 
\begin{equation*}
\xymatrix{
SHI(-M,-\Gamma_1)\ar[rr]^{\Psi_{\beta}}&&SHI(-M,-\Gamma_2)\ar[dl]^{\Psi_{\theta}}\\
&SHI(-M,-\Gamma_3)\ar[lu]^{\Psi_{\eta}}&
}	
\end{equation*}
\end{theorem}

In \cite{li2019direct}, the first author also developed a graded version of \Cref{thm: bypass exact triangle}. Suppose $M$ is a solid torus and $\Gamma_2=(1,-p,q)$ is a dividing set on $\partial M$ so that $(p,q)\neq (1,1)$. We have a continued fraction expansion as in the introduction:
\[-\frac pq=[r_0,r_1,\dots,r_k]=r_0-\frac1{r_1-\frac1{r_2-\dots\frac1{r_k}}},\]
where $r_i\le-2$ are integers. Define
$$-\frac{p_m}{q_m}=[r_0,r_1,\dots,r_k,-m].$$
\begin{theorem}[Honda \cite{honda2000classification}]\label{thm: there are two bypasses}
There are two bypasses on $\Gamma_m=(1,-p_m,q_m)$, which we call $\beta_{+}$ and $\beta_-$, so that after attaching either of them, the resulting dividing sets are both $\Gamma=(1,-p,q)$. Furthermore, the third dividing set involved in the bypass triangle is $\Gamma_{m-1}=(1,-p_{m-1},q_{m-1})$. 
\end{theorem}

\begin{remark}\label{rem: distinguishing the two bypasses}
We can distinguish the two bypasses either by looking at the intersection of their attaching arc with the dividing set $\Gamma_m$ or by looking at the relative Euler classes as discussed in Honda \cite{honda2000classification}. Because we will not need to distinguish one bypass from the other, we do not present such discussion here.
\end{remark}

\begin{theorem}[Li \cite{li2019direct}]\label{thm: graded bypass exact triangle}
If $m\leq 0$, then the maps $\Psi_{\beta_{+}}$ and $\Phi_{\beta_-}$ are both zero. If $m=1$, then we have injective graded maps
$$\Psi_{\beta_+}:SHI(-M,-\Gamma_m,i)\rightarrow SHI(-M,-\Gamma,i-p_0)$$ and
$$\Psi_{\beta_-}:SHI(-M,-\Gamma_m,i)\rightarrow SHI(-M,-\Gamma,i).$$
Furthermore, if $m\geq 2$, then we have two graded bypass exact triangles for any $i\in\mathbb{Z}$, as follows. 
\begin{equation*}
\xymatrix{
SHI(-M,-\Gamma_m,i)\ar[rr]^{\Psi_{\beta_+}}&&SHI(-M,-\Gamma,i-p_{m-1})\ar[dl]^{\Psi_{\theta_+}}\\
&SHI(-M,-\Gamma_{m-1},i)\ar[lu]^{\Psi_{\eta_+}}&
}	
\end{equation*}
\begin{equation*}
\xymatrix{
SHI(-M,-\Gamma_m,i)\ar[rr]^{\Psi_{\beta_-}}&&SHI(-M,-\Gamma,i)\ar[dl]^{\Psi_{\theta_-}}\\
&SHI(-M,-\Gamma_{m-1},i-p)\ar[lu]^{\Psi_{\eta_-}}&
}	
\end{equation*}
Moreover, $\Psi_{\theta_+}$ and $\Psi_{\theta_-}$ are both zero.
\end{theorem}

\section{\texorpdfstring{Bypasses on a solid torus $S^1\times D^2$}{The solid torus}} \label{sec: solid torus} 
\subsection{States and bypasses} 
Let $\xi$ be a tight contact structure on a solid torus $M$ with $\partial M$ convex. Recall that its dividing set $\Gamma_{\partial M}$, which we often simply denote as $\Gamma$ when there is no confusion, can be parametrized as $(n,-p,q)$, where $0<q\le p$ and $\gcd(p,q)=1$, and where we use the orientations shown in \Cref{fig:lambda-mu-torus}. If $\Gamma$ can be parametrized as $(n,-p,q)$, then we abbreviate our notation and say that $\Gamma=(n,-p,q)$. 

\begin{figure}[htbp] 
\centering
\includegraphics[scale=0.75]{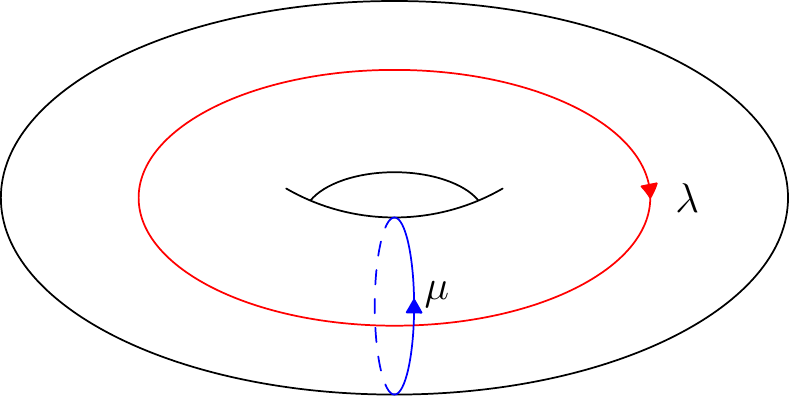}
\caption{We take our orientations for $\lambda$ and $\mu$ to be as shown above. Note that the suture $(n,-p,q)$ means that there are $2n$ components, each of which follows the $\lambda$ direction $-p$ times and the $\mu$ direction $q$ times.}
\label{fig:lambda-mu-torus}
\end{figure} 

Moreover, we can label the components in order along the $\mu$ direction as $\gamma_1,\gamma_2,\dots,\gamma_{2n}$. Also suppose, without loss of generality, that $\gamma_1$ is oriented in the $-\lambda$ and $+\mu$ directions. We can also fix some point $w$ on $\gamma_1$. For any meridian disk $D$, label the points of intersection between $\Gamma=\bigcup\gamma_i$ and $\partial D$ as $P_1,P_2,\dots,P_{2np}$ such that $P_i,P_{2n+i},\dots,_{2n(p-1)+i}\in\gamma_i$, the point $P_1$ is the first intersection of $\partial D$ and $\gamma_1$ if we begin at $w$ and go along the component $\gamma_1$, and the points are ordered in the $\mu$ direction. 

Suppose $(n,-p,q)\neq (1,-1,1)$ and take $D$ to be a meridian disk. We may perturb $\partial D$ to make it Legendrian by the Legendrian realization principle in \Cref{thm: legendrian realization}. Then \Cref{thm: twisting number relative to framing} implies that $t(\partial D)\leq 0.$ Hence \Cref{thm: convex flexibility} applies, and we can further perturb $D$ to be convex. As discussed by Honda \cite{honda2000classification}, since the dividing set $\Gamma_{\partial M}$ intersects $\partial D$ at $2np$ points, it follows that the dividing set $\Gamma_D$ on $D$ also intersects $\partial D$ at $2np$ points, which alternate between the points of $\Gamma\cap \partial D$. According to Giroux's criterion in \Cref{thm: giroux's criterion}, moreover, the dividing set $\Gamma_D$ does not admit any closed components. 

\begin{definition}\label{defn: states}
Suppose $\xi$ is a tight contact structure on a solid torus $M$ with $\partial M$ convex and dividing set $\Gamma=(n,-p,q)$. A convex meridian disk $D$ together with the dividing set $\Gamma_{D}$ on $D$ is called a \emph{state} of $\xi$. To specify the pre-fixed dividing set $\Gamma$ on $\partial M$, we also call it a \emph{$(n,-p,q)$-state}. 
\end{definition}

See \Cref{fig:meridian-dividing-set} for an example of a possible state $\Gamma_D$. 

\begin{figure}[htbp] 
\centering
\includegraphics[scale=0.75]{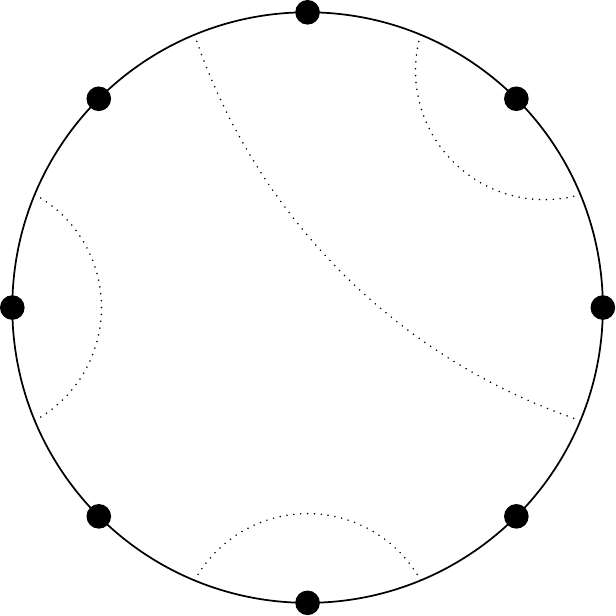}
\caption{The dotted arcs are the dividing curves on $\Sigma=D^2$, while the dots are the points of $\Gamma\cap\partial D$.}
\label{fig:meridian-dividing-set}
\end{figure} 

\begin{remark}
Note that there are a total of $C_{np}$ possible candidates for $(n,-p,q)$-states, though some of them might not admit any tight contact structure. Here $C_{k}$ is the $k$-th Catalan number. 
\end{remark}

Recall that for any given meridian disk $D$, we have labeled the points of $\Gamma\cap\partial D$ from $1$ to $2np$. If $np>1$, then we make the following definition.

\begin{definition}\label{defn: bypass on i}
Suppose $(D,\Gamma_D)$ is a $\Gamma$-state for the contact structure $\xi$ on $(M,\Gamma)$, and $\gamma$ is a boundary-parallel component of $\Gamma_D$ (cf. \Cref{prop: bypass abundance}). If the disk co-bounded by $\gamma$ and part of $\partial D$ intersects $\Gamma$ at the point $P_i$, we say that $\gamma$ is {\it centered at} $P_i$. Furthermore, we also say that the state $(D,\Gamma_D)$ \emph{admits a bypass on $\gamma_j$}, where $j$ is the residue of $i$ modulo $2n$. 
\end{definition}

A priori, a fixed contact structure could admit multiple states which do not necessarily exist simultaneously. However, the following lemma guarantees that this is not possible. 

\begin{lemma}\label{lem: states exists simultaneously}
Suppose $\xi$ is a tight contact structure on a solid torus $M$ with $\partial M$ convex and dividing set $\Gamma=(n,-p,q)$. Suppose $(D_i,\Gamma_{D_i})$ for $1\le i\le k$ are all possible states of $\xi$. Then there exist meridian disks $D_1',D_2',\dots,D_k'$ in $M$ so that the following holds. 

\begin{enumerate}[(i)]
\item $D'_1,D'_2,\dots, D'_k$ are all convex with Legendrian boundary and pairwise disjoint.
\item For each $1\le i\le k$, the dividing set $\Gamma_{D'_i}$ coincide with $\Gamma_{D_i}$.
\end{enumerate} 
\end{lemma}
\begin{proof}
We look at the $(Np+1)$-cover of $(M,\Gamma)$, which we call $(M',\Gamma')$. Here $N$ is a sufficiently large positive integer. Note that $M'$ is still a solid torus, and $\Gamma'=(n,-p,(Np+1)q)$. We can send $\Gamma'$ back to $\Gamma$ by performing Dehn twists along a meridian disk. Let $\xi'$ be the pullback of $\xi$ on $M'$. By construction, a state of $\xi$ is a state of $\xi'$. Because the state determines the contact structure on a solid torus by Honda \cite{honda2000classification}, we see that $\xi$ and $\xi'$ are contactomorphic under the diffeomorphism $(M,\Gamma)\cong(M',\Gamma')$ induced by a Dehn twist along a meridian disk. Note $(M',\xi')$ is the $(Np+1)$-cover of $(M,\xi)$, and we can take $N$ large so that $Np+1\gg k$. Then we can pick different states for $\xi=\xi'$ from different blocks of $M'$.
\end{proof}

Bypasses also give rise to states as in the following lemma.

\begin{lemma}\label{lem: bypasses also give rise to states}
Suppose $M$ is a solid torus and $\Gamma=(n,-p,q)$ is a dividing set so that $n>1$. Let the components of $\Gamma$ be $\gamma_1,\dots,\gamma_{2n}$. Suppose that $\xi$ is a contact structure on $(M,\Gamma)$ admitting a bypass $\alpha$ whose attaching arc intersects $\gamma_{i-1}$, $\gamma_i$, and $\gamma_{i+1}$. Then $\xi$ admits a state that admits a bypass on $\gamma_i$. 
\end{lemma}
\begin{proof}
By \Cref{defn: bypass}, the bypass disk is itself a convex disk whose intersection with $\partial M$ is Legendrian. We can extend the bypass disk to be a meridian disk, and apply \Cref{thm: legendrian realization} and \Cref{thm: convex flexibility} to perturb the meridian disk to be convex with Legendrian boundary. However, when perturbing the meridian disk, the part which have already been Legendrian or convex can be fixed. The lemma then follows.
\end{proof}

\subsection{Parametrizing the contact structures}
For a given potential dividing set $\Gamma=(n,-p,q)$ on a solid torus $M$, we want to parametrize and stratify the set of all (isotopy classes of) tight contact structures according to the states they admit. To do so, we introduce the following notations and definitions. 

\begin{definition}
For a tuple $\vec{x}=(x_1,\dots,x_{2n})\in\{0,1\}^{2n},$ define
$$\norm{\vec{x}}=\sum_{i=1}^{2n}x_i.$$
For two tuples $\vec{x}^1=(x_1^1,\dots,x_{2n}^1)$ and $\vec{x}^2=(x_1^2,\dots,x_{2n}^2)$, define
$$\vec{x}^1\cap \vec{x}^2=(x_1^1\cdot x_1^2,\dots,x_{2n}^1\cdot x_{2n}^2).$$
	
\end{definition}

\begin{definition}
Suppose	$\Gamma=(n,-p,q)$ is a dividing set on a solid torus $M$. Then we define $\mathcal{C}(n,-p,q)$ to be the set of all contact structures on $(M,\Gamma)$. Define $\mathcal{T}(n,-p,q)\subset \mathcal{C}(n,-p,q)$ to be the set of all tight ones, and let
$N(n,-p,q)=|\mathcal{T}(n,-p,q)|$.
\end{definition}

\begin{definition}
Write $\Gamma=\gamma_1\cup\dots\cup\gamma_{2n}$, where $\gamma_i$ are the components of $\Gamma$. For a tuple
$\vec{x}\in\{0,1\}^{2n},$ define
$$\mathcal{T}_{\vec{x}}(n,-p,q)\subset \mathcal{T}(n,-p,q)$$
to be the set of all isotopy classes of tight contact structures $\xi$ such that there is a state $(D,\Gamma_D)$ of $\xi$ which admits a bypass on $\gamma_i$ for all $i$ with $x_i=1$ (cf. \Cref{defn: bypass on i}). An example is shown in \Cref{fig:tx-def}.
\end{definition} 

\begin{figure}[htbp]
\centering
\includegraphics[width=0.7\textwidth]{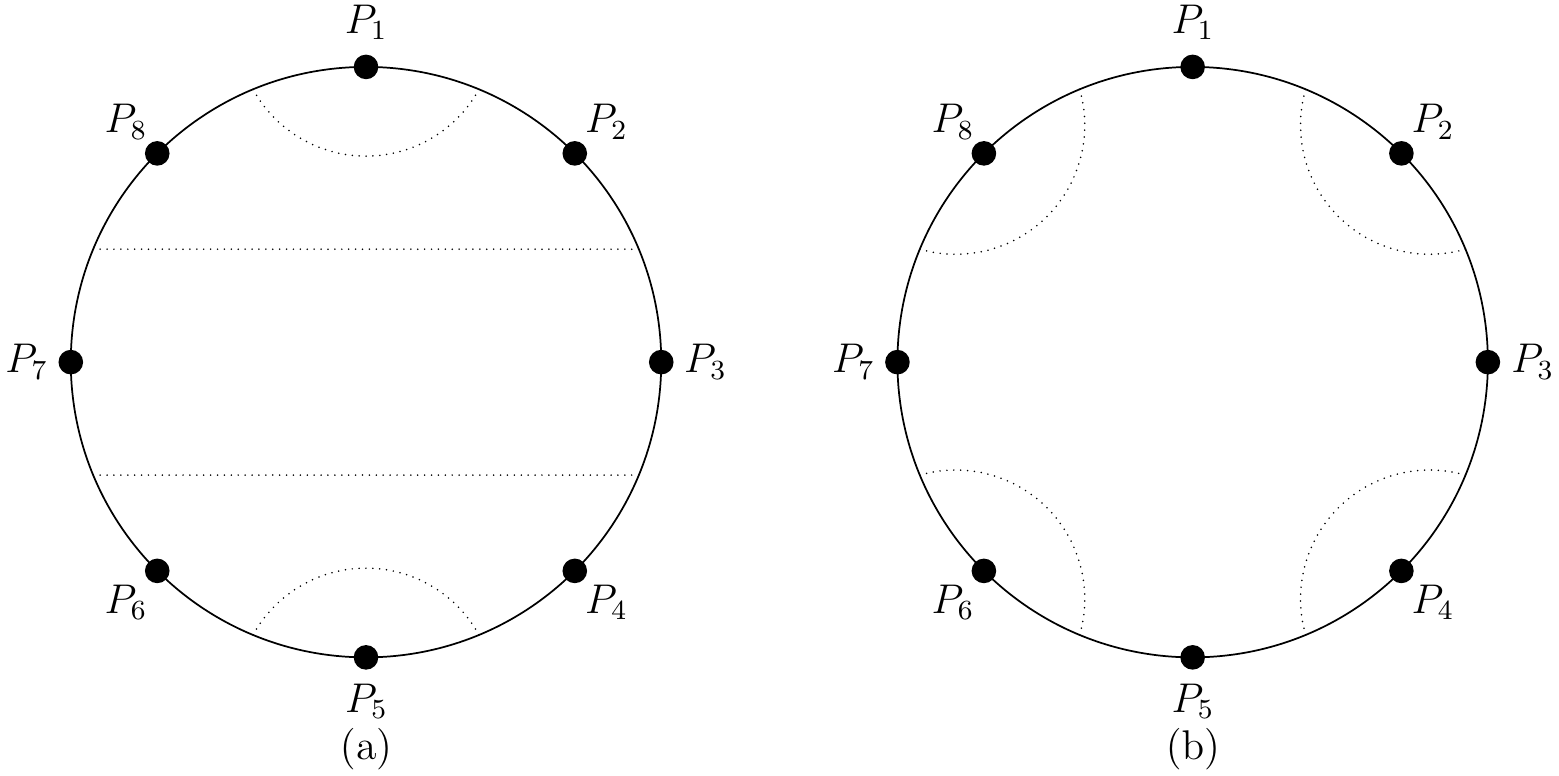}
\caption{If a tight contact structure $\xi$ with dividing set $(2,-2,1)$ admits the above two states, then it would be an element of $\mathcal T_{\vec x}(2,-2,1)$ for every $\vec x\in\{0,1\}^2$. If, however, it only admitted (a), then $\xi$ would not be an element of $\mathcal T_{(0,1)}$ as the only boundary-parallel bypasses admitted by the disk (a) are centered on $\gamma_1$.}
\label{fig:tx-def}
\end{figure}

From the definition, it is clear that
\begin{equation}\label{eq: inclusion-exclusion 1}
	\mathcal{T}(n,-p,q)=\bigcup_{\vec{x}\in\{0,1\}^{2n}}\mathcal{T}_{\vec{x}}(n,-p,q).
\end{equation}
Moreover, if $\vec{x}^1$ and $\vec{x}^2$ are two tuples, then
\begin{equation}\label{eq: inclusion-exclusion 2}
	\mathcal{T}_{\vec{x}^1}(n,-p,q)\cap \mathcal{T}_{\vec{x}^2}(n,-p,q)=\mathcal{T}_{\vec{x}^1\cap\vec{x}^1}(n,-p,q).
\end{equation}
Hence to utilize the inclusion-exclusion principle to calculate $N(n,-p,q)=|\mathcal{T}(n,-p,q)|$, we need to understand $|\mathcal{T}_{\vec{x}}(n,-p,q)|$ for all possible $\vec{x}\in\{0,1\}^{2n}$. We first tackle the case when $\norm{\vec x}>n$. 

\begin{lemma} \label{lem: disallow adjacent bypasses}
Suppose $n>1$. Let $\vec{x}$ be a tuple $(x_1,\dots,x_{2n})\in\{0,1\}^{2n}$ such that there exists $i\in\{1,\dots,2n\}$ with $x_i=x_{i+1}=1$, where $x_{2n+1}$ is the same as $x_1$. Then
$$\mathcal{T}_{\vec{x}}(n,-p,q)=\emptyset.$$
\end{lemma}

\begin{proof}
Suppose that $\xi\in\TT_{\vec x}(n,-p,q)$ for such a tuple $\vec x$. Suppose without loss of generality that $i=1$, so that $x_1=x_2=1$. Note that \Cref{lem: states exists simultaneously} implies that there exist states $(D_1,\Gamma_{D_1})$ and $(D_2,\Gamma_{D_2})$ such that $D_i$ admits a bypass on $\gamma_i$ for $i=1,2$ (c.f. \Cref{defn: bypass on i}). Say that the bypasses are $\alpha$ and $\beta$, and are centered on $P_a$ and $P_b$, respectively. 

Now if we were to attach both bypasses, we would be taken from a contact structure $\xi$ on $\Gamma=(n,-p,q)$ to a contact structure $\xi'$ on $\Gamma'$. 

The bypass attachment lemma tells us that the dividing set only changes at the points $P_a$ and $P_{a\pm1}$ on $D_1$, and at the points $P_b$ and $P_{b\pm1}$ on $D_2$. We call these six points the ``affected points.'' 

Consider the component $\gamma\subset\Gamma'$ which contains the point $a$ on $D_1$ (or, more accurately, contains some point on $\gamma_1$ to the left of $P_a$). Note that the point $P_{a+1}$ also lies on $\gamma$ as well. 

Since $n\ne1$, observe that as we traverse along $\gamma_1$, the only affected point we will hit is $P_b$ on $D_2$, which we will reach after $k$ twists along the meridian. This will take us to $P_{b-1}$ on $D_2$, which will take us back to $P_a$ after $-k$ twists along the meridian. Thus $\gamma$ in this case is a homotopically trivial curve. 




We know, furthermore, that attaching an interior bypass corresponds to taking a particular subset of $M$. As such, attaching bypasses to a tight contact structure could not possibly create an overtwisted disk, and so the resulting contact structure must also be tight. However, the subset of $M$ that is created by attaching both the bypass at 1 and the bypass at 2 does \emph{not} have a tight neighborhood by Giroux's criterion in \Cref{thm: giroux's criterion}. After all, we already know that $\Gamma'$ has a homotopically trivial loop. It thus follows that the original contact structure could not have been tight at all. 

We therefore conclude that $\TT_{\vec x}(n,-p,q)=\emptyset$, as desired. 
\end{proof} 

An example of the contradiction that arises when attempting to attach \Cref{fig:adjacent-contradiction} implies that tight contact structures cannot admit adjacent bypasses. Note that adjacency here refers to the adjacency of the central components (of the dividing set), not to the adjacency of the central points (of the intersection of the dividing set with the boundary of the meridian disk). 

\begin{figure}[htbp]
\centering
\includegraphics[scale=0.8]{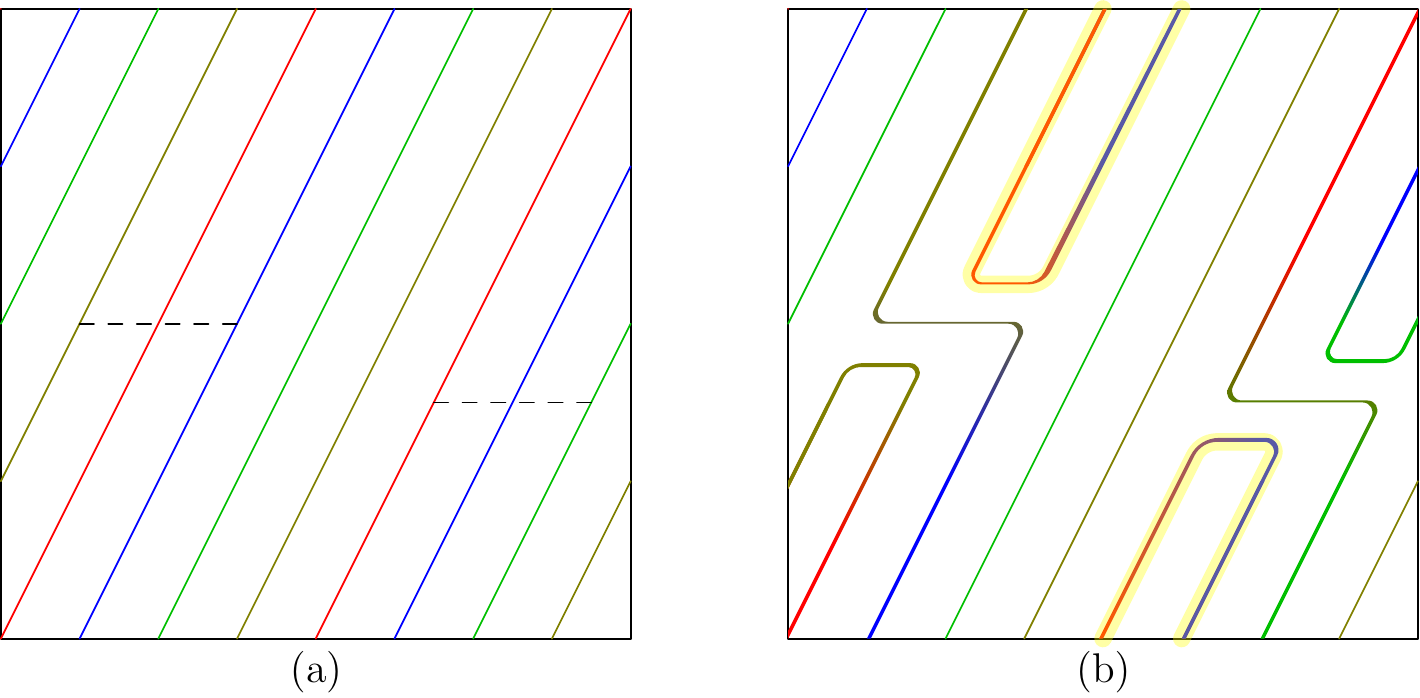}
\caption{The curve highlighted in yellow is an example of a homotopically trivial curve on the torus that is created upon the attachment of two boundary-parallel bypasses centered adjacent components on the dividing set $(2,-2,1)$.}
\label{fig:adjacent-contradiction}
\end{figure}

\begin{corollary}
Suppose $n>1$. For a tuple $\vec{x}\in\{0,1\}^{2n}$ so that $\norm{\vec{x}}>n$, we know
that
$$\mathcal{T}_{\vec{x}}(n,-p,q)=\emptyset.$$
\end{corollary}

Next, we deal with tuples $\vec{x}\in\{0,1\}^{2n}$ so that $\norm{\vec{x}}\leq n$. The case $\norm{\vec{x}}< n$ is relatively simple. It requires the following lemma, however. 

\begin{lemma} \label{lem: injective map}
Suppose $n>1$. Suppose $\xi$ is a tight contact structure on $M=S^1\times D^2$. Let $\Gamma=(n,-p,q)$ be the parametrization of the dividing set. For each $\alpha=0,1,\dots,2np-1$ such that, with the contact structure $\xi$, there is a bypass induced by a boundary-parallel dividing curve centered at $\alpha$, there is a corresponding injective map \[B_\alpha:\mathrm{Tight}(M,(n-1,-p,q))\to\mathrm{Tight}(M,(n,-p,q))\] obtained by attaching the interior bypass put upside down.  
\end{lemma}

\begin{proof}
We know that attaching this interior bypass $\alpha$ results in Case (2) of \Cref{prop: bypass attachment possibilities}. In particular, it takes $\xi$ on $(n,-p,q)$ to a tight contact structure on $(n-1,-p,q)$. Moreover, by \Cref{prop: reverse bypass}, there is a corresponding exterior bypass $\beta$. Since we can attach this exterior bypass to any tight contact structure on $(n-1,-p,q)$, it follows that this induces a map $B_\alpha$ from $\mathcal T(n-1,-p,q)$ to $\mathcal C(n,-p,q)$. 

It thus remains to show that $\im B_\alpha\subseteq\mathcal T(n-1,-p,q)$ and that $B_\alpha$ is injective. To do this, consider an arbitrary tight contact structure $\xi$ on $(n-1,-p,q)$. 

Note that if we let the horizontal direction correspond to $-p\lambda+q\mu$, then we have a total of $2n$ horizontal components of $(n,-p,q)$. Number them $\gamma_0$ through $\gamma_{2n-1}$, and name $k$ to be the point of intersection of $\gamma_k$ with the vertical edge. Moreover, suppose without loss of generality that $k=1$, so $\alpha$ is centered at $\gamma_1$. 

Then it is clear the $\gamma_0$, $\gamma_1$, and $\gamma_2$ are all joined up into some component $\gamma$. Then $\beta$ is a bypass, all of whose points of intersection with the dividing set are on $\gamma$. Moreover, we can choose a bypass attaching arc $\eta$ with two points of intersection with $\gamma$ and one point of intersection with $\gamma_3$, as shown in \Cref{fig:beta-bypass}. 

\begin{figure}[htbp] 
\centering
\includegraphics[width=0.7\textwidth]{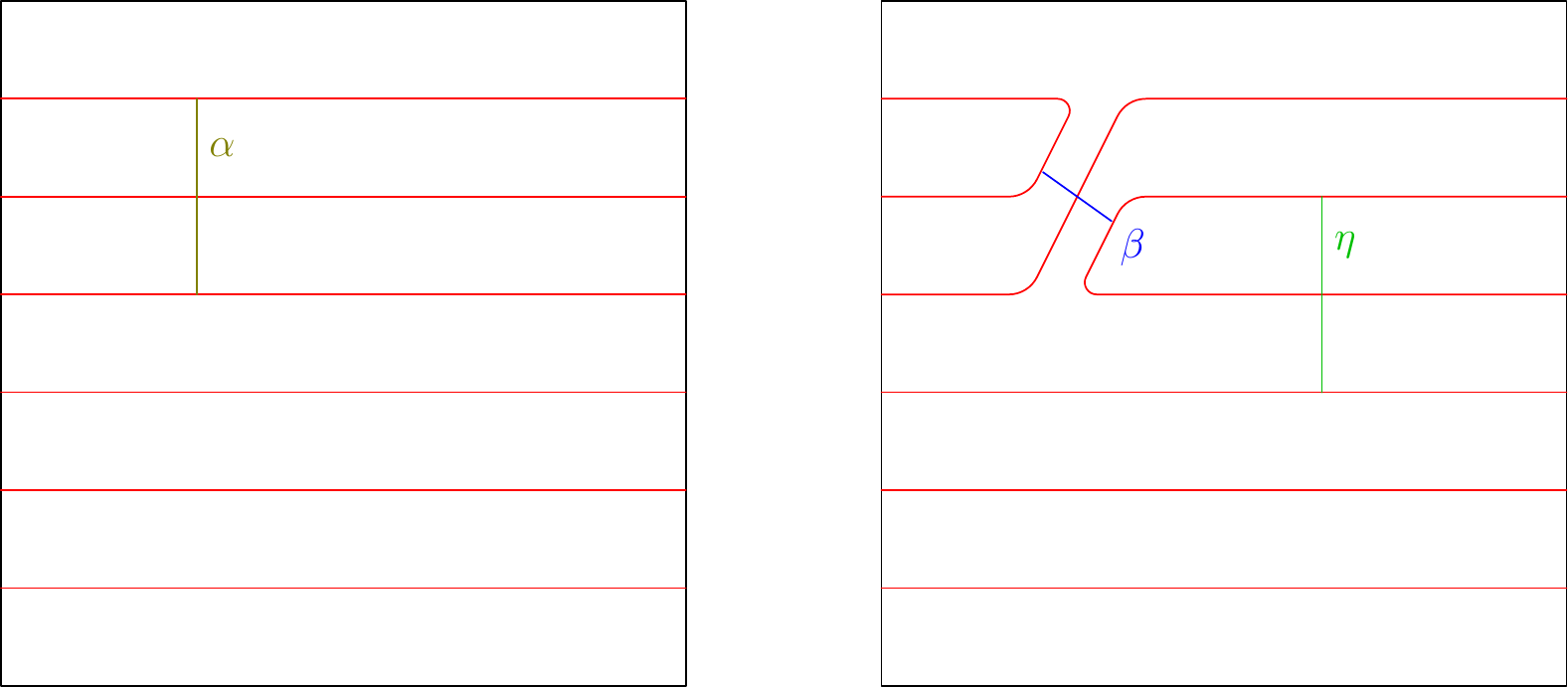}
\caption{We begin by (a) attaching $\alpha$ to the dividing set $(n,-p,q)$. Then we can (b) identify the bypasses $\beta$ and $\eta$ on $(n-1,-p,q)$. Note that in this diagram, the horizontal direction corresponds to $-p\lambda+q\mu$.}
\label{fig:beta-bypass}
\end{figure}

By \Cref{prop: bypasses commute} and the fact that $\beta$ and $\eta$ have disjoint attaching arcs, we have that $\beta\circ\eta=\eta\circ\beta$, which is seen in \Cref{fig:bypass-cd}. 

\begin{figure}[htbp] 
\centering
\includegraphics[width=0.7\textwidth]{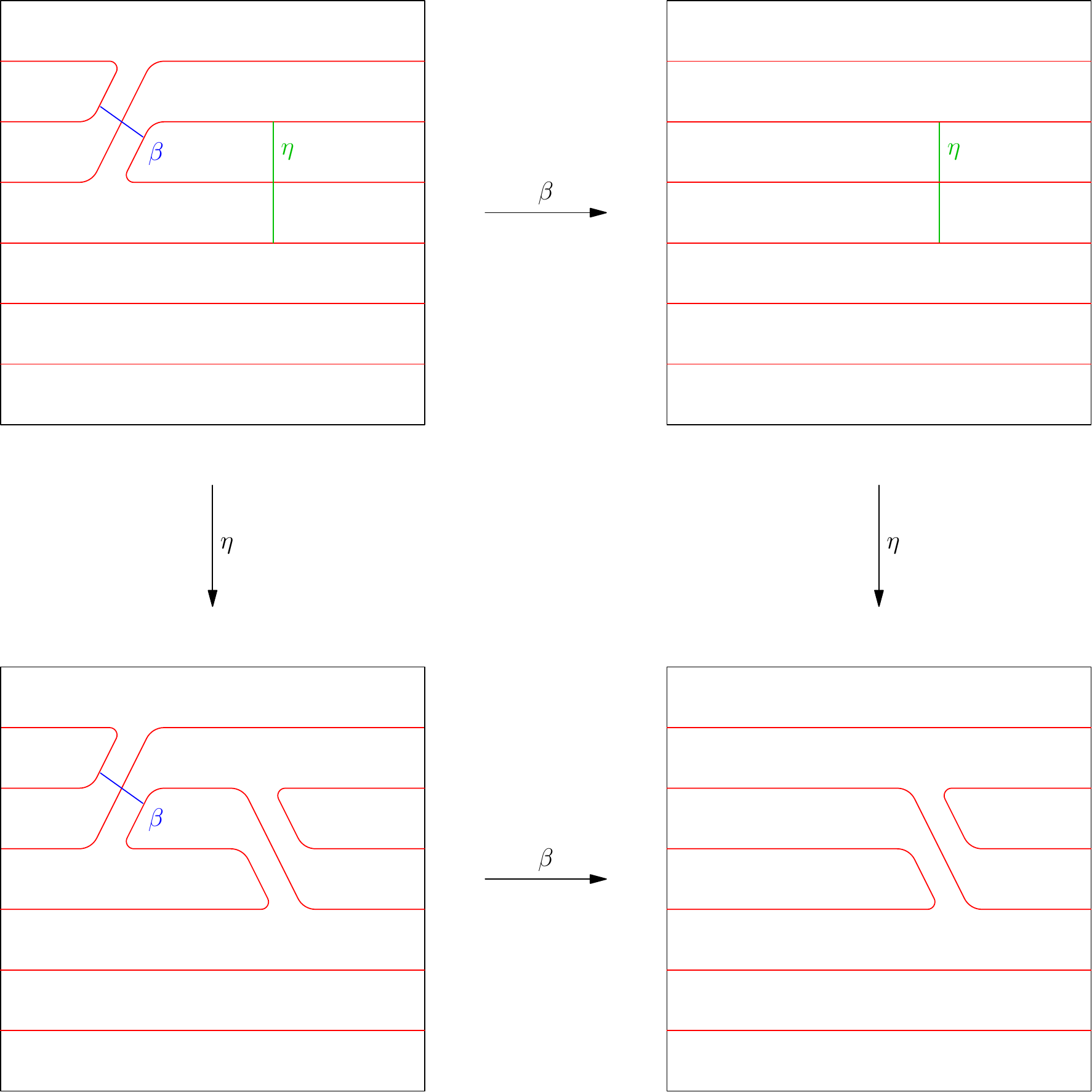}
\caption{With $\beta$ and $\eta$ exterior bypasses with attaching arcs as shown, the above diagram will always commute. Note that the bypasses in $\beta\circ\eta$ are both trivial.}
\label{fig:bypass-cd}
\end{figure}

But observe that $\eta$ is a trivial bypass on $(n-1,-p,q)$ if we attach it first, so by \Cref{prop: trivial bypasses are trivial}, first attaching $\eta$ does not change the dividing set. Then $\beta$ becomes a trivial bypass on $(n-1,-p,q)$ after attaching $\eta$. It follows that $\eta\circ\beta=\beta\circ\eta$ does not change the contact structure either and is in this way trivial. 

In other words, we know that the function $\eta\circ\beta:\mathcal T(n-1,-p,q)\to\mathcal C(n-1,-p,q)$ is injective and has $\im(\eta\circ\beta)=\mathcal T(n-1,-p,q)$. It then follows immediately that $\beta$ must be injective with $\im\beta\subseteq\mathcal T(n,-p,q)$. The latter conclusion follows once we recall that $\eta$ is an exterior bypass. In particular, if $\beta(\xi')$ is overtwisted, where $\xi'$ is a tight contact structure on $(n-1,-p,q)$, then $\xi'=\eta(\beta(\xi'))$, which would correspond to taking a neighborhood of $\beta(\xi')$, must contain this overtwisted disk as well, a contradiction. 
\end{proof}

This allows us to deal with the case $\norm{\vec x}=k<n$, which will be essential to make our inclusion-exclusion work. Note that we separate the $k<n$ case from the $k=n$ case because they require different arguments. In particular, the latter case will involve a slight subtlety arising from the fact that attaching the $n$-th bypass may result in some undetermined dividing sets. For the $k<n$ case, which we prove now, however, we can always decide precisely what the dividing set after attaching $k$ bypasses will be, which eliminates the difficulty found in the $k=n$ case.

\begin{prop} \label{prop: tight contact structures admitting k bypasses} 
Suppose $n>1$. For a tuple $\vec{x}\in\{0,1\}^{2n}$ such that $\norm{\vec{x}}=k<n$ and such that there does not exist an $i$ with $x_i=x_{i+1}=1$, we know
that
$$|\mathcal{T}_{\vec{x}}(n,-p,q)|=|\mathcal{T}(n-k,-p,q)|=N(n-k,-p,q).$$
\end{prop}

\begin{proof}
It suffices to show that if we have multiple bypasses that are each induced by a boundary-parallel curve on the states of a contact structure with dividing set $(n,-p,q)$, then after applying Case (2) of \Cref{prop: bypass attachment possibilities} to one of them, we can apply it to the new contact manifold to another of them. That is to say, we would like to show that when we attach a bypass induced by a boundary-parallel curve, each of the other bypass disks is boundary-parallel on the new dividing set. 

First, we will show that we can attach all $k$ bypasses and will arrive at a contact structure with dividing set $(n-k,-p,q)$. Note that if we can prove this in the case $k=2$, then induction will prove the general case. Thus suppose that $x_m=x_{m'}=1$ for some nonadjacent $m$ and $m'$. In particular, suppose that we can attach the bypasses $\alpha$ and $\beta$ at the points $a\equiv m\pmod{2n}$ and $b\equiv m'\pmod{2n}$, respectively. 

We already know that attaching $\alpha$ will take us to $(n-1,-p,q)$. Moreover, its only effect is to connect the components $\gamma_{a-1}$, $\gamma_a$, and $\gamma_{a+1}$ into one component of $\Gamma_{\partial M}=(n-1,-p,q)$, which we can call $\gamma'_a$. The bypass centered at $b$ has attaching arc intersecting $\gamma_{b-1}$, $\gamma_b$, and $\gamma_{b+1}$. 

By the nonadjacency condition, the only possible overlap with the components affected by $\alpha$ is if $\gamma_{b+1}=\gamma_{a-1}$ (or, symmetrically, if $\gamma_{b-1}=\gamma_{a+1}$). 

Note that if there is no overlap between the two sets of components, then the attaching arc of the bypass $\beta$ is unaffected. Even if there is overlap, however, we know that $\beta$ is attached at an arc intersecting $\gamma_{b-1}$, $\gamma_b$, and $\gamma_{a+1}=\gamma'_a$. Since $\gamma'_a$ is still adjacent to $\gamma_b$, this attaching arc still corresponds to a boundary-parallel curve. This means that attaching $b$ will take us to a contact structure whose dividing set is $(n-2,-p,q)$. Thus, in general, we can say that attaching $k<n$ bypasses centered on nonadjacent components of $\Gamma=(n,-p,q)$ will take us to $(n-k,-p,q)$.

But \Cref{lem: injective map} implies that we can then reverse the bypasses. This gives us an injective map $\TT(n-k,-p,q)\hookrightarrow\TT(n,-p,q)$ whose image is exactly those elements in $\TT_{\vec x}(n,-p,q)$. The proposition follows. 
\end{proof}

Finally, we deal with the remaining case when $\norm{\vec{x}}=n$. Recall our definitions in \Cref{subsec: floer theory} of $p_m$ and $q_m$. Moreover, if $(p,q)\neq (1,1)$, then let $-\frac pq=[r_0,r_1,\dots,r_k]$ and we define 
$$s=|(r_0+1)(r_1+1)\dots(r_k+1)|.$$
Then we have the following.

\begin{prop} \label{prop: tight contact structures admitting n bypasses}
	Suppose $\vec{x}\in\{0,1\}^{2n}$ is such that $\norm{\vec{x}}=n$.
	
	(1) If $(p,q)=(1,1)$, then we know that 
	$$|\mathcal{T}_{\vec{x}}(n,-p,q)|=1.$$
	
	(2) If $(p,q)\neq (1,1)$, then we know that 
	$$|\mathcal{T}_{\vec{x}}(n,-p,q)|=|\mathcal{T}_{\vec{x}}(1,-p_1,q_1)|=s.$$
\end{prop}
\begin{proof}
For Case (1), according to Honda \cite{honda2000classification2}, we know that for the dividing set $\Gamma=(n,-1,1)$, the set of isotopy classes of tight contact structures is in one-to-one correspondence to the set of all possible states. Hence there is a unique state associated to each tight contact structure. Note that
$$|\partial D\cap \Gamma_D|=|\partial D\cap \Gamma|=2n,$$
so if the dividing set of a state consists of $n$ boundary-parallel components, then there are only two possibilities, as drawn in \Cref{fig:n-bound-parallel}, proving this case. 

\begin{figure}[htbp] 
\centering
\includegraphics[width=0.6\textwidth]{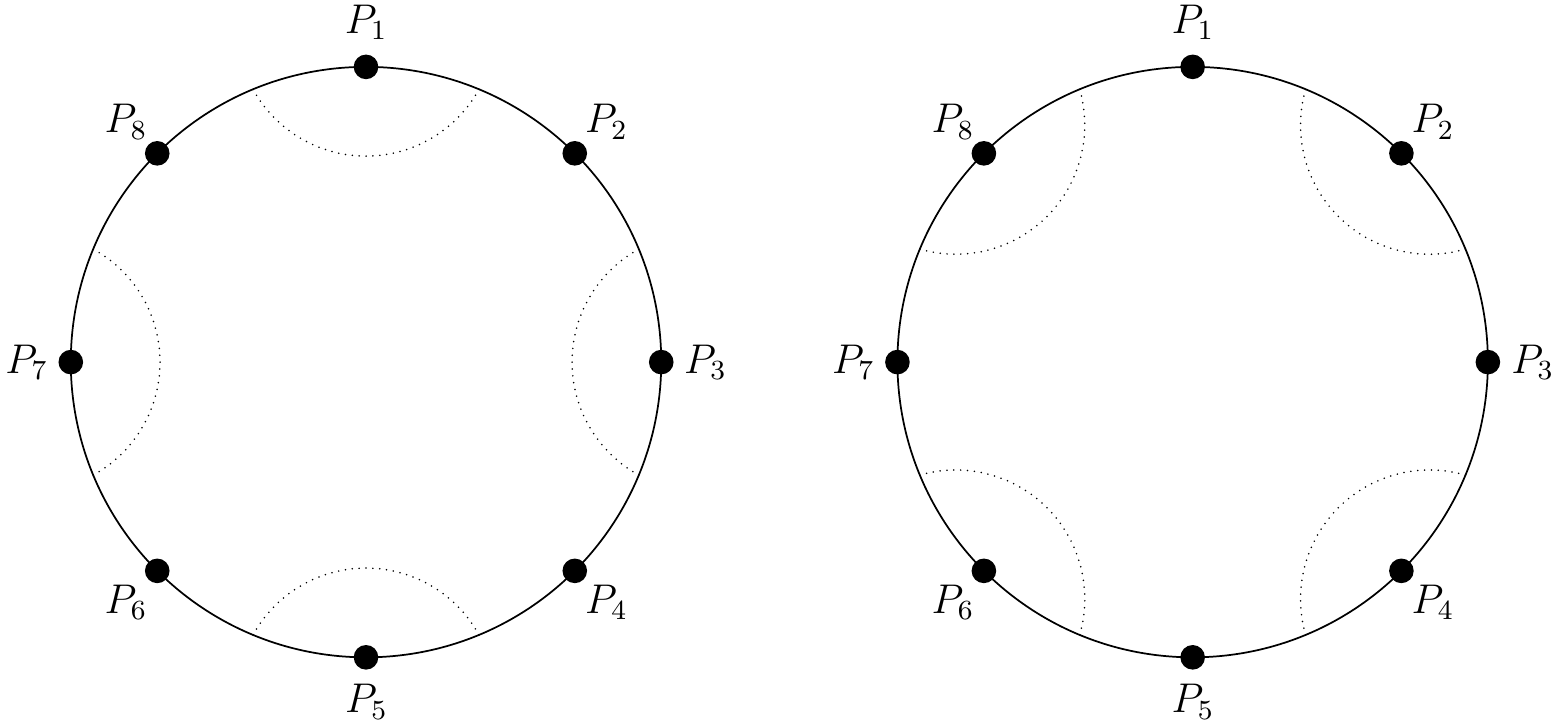} 
\caption{There are only two ways to have $n$ boundary-parallel components when $\Gamma=(n,-1,1)$.}
\label{fig:n-bound-parallel}
\end{figure}

For Case (2), suppose $\xi\in\mathcal{T}_{\vec{x}}(n,-p,q)$ for some tuple $\vec{x}$ satisfying the hypothesis of the proposition. Then $\xi$ admits $n$ different bypasses $\alpha_1,\dots,\alpha_n$. The first $(n-1)$ bypasses will alter the contact structure as in \Cref{prop: tight contact structures admitting k bypasses}: After peeling of layers of the boundary containing $\alpha_1,\dots,\alpha_{n-1}$, we obtain a triple $(M',\Gamma',\xi')$, where $M'$ is still a solid torus, $\Gamma'=(1,-p,q)$, and $\xi'=\xi|_{M'}$ is a tight contact structure on $(M',\Gamma')$. Let $\beta_1,\dots,\beta_n$ be the bypasses obtained by putting $\alpha_1,\dots,\alpha_n$ upside down, respectively, as in \Cref{prop: reverse bypass}, then we know that $\xi$ is obtained from $\xi'$ by attaching all $\beta_{1},\dots,\beta_{n-1}$. Thus we need only to understand the last bypass $\alpha_n$, viewed as a bypass in $(M',\Gamma',\xi')$. The way $\alpha_n$ changes the dividing set $\Gamma'$ is described by term (4) in \Cref{prop: bypass attachment possibilities}, though, a priori, we do not know along which curve to perform the Dehn twist. 
	
According to Honda \cite{honda2000classification}, all possibilities can be described as follows: If we peel of a layer of $\partial M'$ containing the bypass $\alpha_n$, then we obtain a triple $(M'',\Gamma'',\xi'')$, where $M''$ is still a solid torus, $\Gamma''=(1,-p_m,q_m)$ for some $m\in\mathbb{Z}$ (not necessarily positive), and $\xi''=\xi'|_{M''}$. To specify the different possibilities, we write $(M''_m,\Gamma''_m,\xi''_m)$ and let the corresponding bypass be $\alpha_{n,m}$ and $\beta_{n,m}$. Recall, as in the proof of \Cref{lem: injective map}, we use the notations $\beta_{n,m}$ for bypasses to also denote the maps between set of contact structures:
$$\beta_{n,m}:\mathcal{T}(1,-p_m,q_m)\rightarrow \mathcal{C}(1,-p,q).$$
Note by \Cref{thm: there are two bypasses}, there are two possible candidate for the bypasses $\beta_{n,m}$ for all $m$, though, as explained in \Cref{rem: distinguishing the two bypasses}, they can be distinguished by looking at the component of the dividing set $\Gamma''_m$ on which the end points of the bypass arc $\beta_{n,m}$ lie, and hence is determined by the component of the dividing set $\Gamma'$ on which the end point of the bypass arc $\alpha_n$ lie. This latter datum is further encoded in the tuple $\vec{x}\in\{0,1\}^{2n}$. Hence, for a fixed tuple $\vec{x}\in\{0,1\}^{2n}$, whether $\beta_{n,m}$ is a positive bypass or a negative bypass has been predetermined. Since the argument for the two cases are similar, we only work with the case when $\beta_{n,m}$ is a negative bypass for all $m\in\mathbb{Z}$: $\beta_{n,m}=\beta_{n,m,-}$.

\begin{claim} \label{claim 1}
We have ${\rm im}(\beta_{n,m})\cap \mathcal{T}(1,-p,q)=\emptyset$ for $m\leq0$.
\end{claim} 

\begin{proof}[Proof of claim] 
Indeed, from \Cref{thm: graded bypass exact triangle}, we know that the map $\Psi_{\beta_{n,m}}=0$. Then \Cref{claim 1} follows from \Cref{thm: SHI for sutured solid torus} and \Cref{thm: bypass induces maps on SHI}.
\end{proof}

\begin{claim} \label{claim 2}
We have ${\rm im}(\beta_{n,m})\cap \mathcal{T}(1,-p,q)={\rm im}(\beta_{n,m-1})\cap \mathcal{T}(1,-p,q)$ if $m\geq 3$.
\end{claim}

\begin{proof}[Proof of claim] 
Recall we have assumed that $\beta_{n,m}$ is a negative bypass. Let $\beta_{n,m,+}$ be the corresponding positive bypass which also change the dividing set from $\Gamma_m''$ to $\Gamma'$. As explained in \Cref{subsec: floer theory}, $\beta_{n,m,+}$ is involved in a bypass triangle where the other two bypasses are denoted by $\theta_{n,m,+}$ and $\eta_{n,m,+}$. From \Cref{thm: graded bypass exact triangle}, we know that the map
$$\Psi_{\beta_{n,m-1}}:SHI(-M,-\Gamma_{m-1}'')\rightarrow SHI(-M,-\Gamma')$$
is zero for gradings $i\geq p$ and is an isomorphism for gradings $1\leq i\leq p$, the map
$$\Psi_{\beta_{n,m}}:SHI(-M,-\Gamma_{m}'')\rightarrow SHI(-M,-\Gamma')$$
is zero for gradings $i\geq p$ and is an isomorphism for gradings $1\leq i\leq p$, and the map
$$\Psi_{\eta_{n,m-1,+}}:SHI(-M,-\Gamma_{m-1}'')\rightarrow SHI(-M,-\Gamma''_m)$$
is an isomorphism for gradings $1\leq i\leq p$. Note $\Gamma'=(1,-p,q)$ so by \Cref{thm: SHI for sutured solid torus}, $1,\dots,p$ are all the nonzero gradings of $SHI(-M,-\Gamma')$. Also, since $m\geq 3$, we know that $p_m>p$ and $p_{m-1}>p$. Hence \Cref{claim 2} follows from \Cref{thm: SHI for sutured solid torus} and \Cref{thm: bypass induces maps on SHI}.
\end{proof} 

\begin{claim} \label{claim 3}
We have ${\rm im}(\beta_{n,2})\cap \mathcal{T}(1,-p,q)={\rm im}(\beta_{n,1})\cap \mathcal{T}(1,-p,q)$.
\end{claim}

\begin{proof}[Proof of claim] 
In this case, as in the discussion for \Cref{claim 2}, we know that the map
$$\Psi_{\beta_{n,1}}:SHI(-M,-\Gamma_{1}'')\rightarrow SHI(-M,-\Gamma')$$
is an isomorphism for gradings $1\leq i\leq p_1$, the map
$$\Psi_{\beta_{n,2}}:SHI(-M,-\Gamma_{2}'')\rightarrow SHI(-M,-\Gamma')$$
is zero for gradings $i\geq p$ and is an isomorphism for gradings $1\leq i\leq p$, and the map
$$\Psi_{\eta_{n,1,+}}:SHI(-M,-\Gamma_{1}'')\rightarrow SHI(-M,-\Gamma''_2)$$
is an isomorphism for gradings $1\leq i\leq p_1$. This time we cannot directly conclude \Cref{claim 3}, since $p_1<p$. Thus it remains to show that there is no tight contact structures $\xi$ on $(M,\Gamma''_2)$ whose contact elements lie in gradings $p_1<i\leq p$.

To prove this last statement, note that Honda in \cite{honda2000classification} classified all tight contact structures on $(M,\Gamma''_1)$ and $(M,\Gamma_2'')$. The proof of his classification theorem states that
$$\mathcal{T}(1,-p_2,q_2)=\eta_{n,1,+}(\mathcal{T}(1,-p_1,q_1))\cup\eta_{n,1,-}(\mathcal{T}(1,-p_1,q_1)).$$
Here $\eta_{n,1,\pm}$ are the bypasses involved in the bypass triangles associated to $\beta_{n,1,\pm}$, as discussed in \Cref{subsec: floer theory}. Note also we have assumed that $\beta_{n,1}=\beta_{n,1,-}$. From \Cref{thm: graded bypass exact triangle}, we know that ${\rm im}(\Psi_{\beta_{n,1,-}})$ is supported in grading $p<i\leq p_2$, and ${\rm im}(\Psi_{\beta_{n,1,+}})$ is supported in grading $0<i\leq p_1$. Thus there is no tight contact structure $\xi$ on $(M,\Gamma''_2)$ whose contact element lies in gradings $p_1<i\leq p$, proving \Cref{claim 3}.
\end{proof}

From \Cref{claim 1,claim 2,claim 3} and \Cref{thm: SHI for sutured solid torus,thm: bypass induces maps on SHI,thm: graded bypass exact triangle}, we know that there is a bijection between ${\rm im}(\beta_n)\cap \mathcal{T}(1,-p,q)$ and $\mathcal{T}(1,-p_1,q_1)$. Recall that if we put the bypasses $\alpha_1,\dots,\alpha_{n-1}$ upside down, we obtain bypasses $\beta_1,\dots,\beta_{n-1}$. From the proof of \Cref{lem: injective map}, we know that 
$$\beta_{1}\circ\dots\beta_{n-1}:\mathcal{T}(1,-p,q)\rightarrow\mathcal{C}(n,-p,q)$$
is injective and has image contained in $\mathcal{T}(n,-p,q)$. Thus, we conclude that
$$\mathcal{T}_{\vec{x}}(n,-p,q)=\beta_{1}\circ\dots\beta_{n-1}({\rm im}(\beta_n)\cap \mathcal{T}(1,-p,q)),$$
and hence
$$|\mathcal{T}_{\vec{x}}(n,-p,q)|=|\mathcal{T}(1,-p_1,q_1)|=s.$$
This concludes the proof of \Cref{prop: tight contact structures admitting n bypasses}.
\end{proof}

\section{Proof of main theorem} \label{sec: proof} 
With this parametrization of the contact structures, we can prove the main theorem, which classifies all tight contact structures on the solid torus based on their dividing sets on $\partial M$. Our first step is to use the method of bypasses to reduce $N(n,-p,q)$, the number of tight contact structures with dividing set $(n,-p,q)$, to previous cases. We refer to this technique as that of ``bypass induction.'' Note that the base case in this bypass induction, namely when $n=1$, is already solved in \Cref{thm: 1pq case}. 

\begin{lemma} \label{lem: N(npq) recurrence} 
For every positive integer $n$, we have \begin{equation}\label{eq: N(npq) recurrence}N(n,-p,q)=\sum_{k=1}^n(-1)^{k+1}\left[\binom{2n-k}k+\binom{2n-k-1}{k-1}\right]N(n-k,-p,q),\end{equation} where $N(0,-p,q)$ is shorthand for $s=|(r_0+1)(r_1+1)\dots(r_k+1)|$. 
\end{lemma}

\begin{proof}
Notice that because of \Cref{eq: inclusion-exclusion 1,eq: inclusion-exclusion 2}, which state that \[\TT(n,-p,q)=\bigcup_{\vec x\in\{0,1\}^{2n}}\TT_{\vec x}(n,-p,q)\] and that \[\TT_{\vec x^1}(n,-p,q)\cap\TT_{\vec x^2}(n,-p,q)=\TT_{\vec x^1\cap\vec x^2}(n,-p,q),\] we can use the principle of inclusion-exclusion. In particular, we have that \[|\TT(n,-p,q)|=\sum_{k=1}^n\left(\sum_{\lVert\vec x\rVert=k}(-1)^{k+1}|\TT_{\vec x}(n,-p,q)|\right).\]

This, along with \Cref{lem: disallow adjacent bypasses} and \Cref{prop: tight contact structures admitting k bypasses,prop: tight contact structures admitting n bypasses}, shows that we must have the recurrence \[N(n,-p,q)=\sum_{k=1}^n(-1)^{k+1}a_{k,n}N(n-k,-p,q),\] where $a_{k,n}$ is the number of tuples $\vec x\in\{0,1\}^{2n}$ without adjacent 1's (where $x_1$ and $x_{2n}$ are considered to be adjacent). But Kaplansky \cite{kaplansky1943nonadj} showed that \[a_{k,n}=\binom{2n-k}k\cdot\frac n{n-k}=\binom{2n-k}k+\binom{2n-k-1}{k-1}.\] Thus \Cref{eq: N(npq) recurrence} holds. 
\end{proof} 

With \Cref{lem: N(npq) recurrence}, we now have a recurrence for $N(n,-p,q)$ in terms of $N(n-k,-p,q)$ for $k=1,\dots,n$. To complete the proof of the theorem, it suffices to find a closed form for the recurrence obtained in \Cref{lem: N(npq) recurrence}. To do this, it will be helpful to first obtain several recurrences involving the Catalan numbers. 

\begin{lemma} \label{lem: C_n recurrence} 
For every positive integer $n$, we have \begin{equation*}C_n=\sum_{k=1}^n(-1)^{k+1}\left[\binom{2n-k}k+\binom{2-k-1}{k-1}\right]C_{n-k}.\end{equation*} 
\end{lemma}

\begin{proof}
We know by \Cref{thm: n10 case} that there are precisely $C_n$ tight contact structures with $\Gamma=(n,-1,1)$. \Cref{lem: N(npq) recurrence} implies the result. 
\end{proof}

\begin{lemma} \label{lem: nC_n recurrence}
For every positive integer $n$, we have \begin{equation}\label{eq: nC_n recurrence} C_n(n+1)=\sum_{k=1}^n(-1)^{k+1}\left[\binom{2n-k}k+\binom{2n-k-1}{k-1}\right]C_{n-k}(n-k+1).\end{equation} 
\end{lemma} 

\begin{proof}
Observe that $C_n(n+1)=\binom{2n}n$ and the term in the brackets is simply 1 when $k=0$. Thus it suffices, in fact, to simply prove that \[\sum_{k=0}^n(-1)^k\left[\binom{2n-k}k+\binom{2n-k-1}{k-1}\right]\binom{2n-2k}{n-k}=0.\] But we know that \begin{align*}\left[\binom{2n-k}k+\binom{2n-k-1}{k-1}\right]\binom{2n-2k}{n-k}&=\frac{2n(2n-k-1)!}{k!(n-k)!^2}\\&=2n\binom nk\binom{2n-k}n\cdot\frac1{2n-k}.\end{align*}

Thus to show \Cref{lem: nC_n recurrence}, it is sufficient to show that \[\sum_{k=0}^n(-1)^k\binom nk\binom{2n-k}n\cdot\frac1{2n-k}=0.\] But because $\binom nk=\binom n{n-k}$, by replacing $k$ with $n-k$, we find that it is sufficient to show that \begin{equation*}S_n\coloneqq\sum_{k=0}^n(-1)^{n-k}\binom n{n-k}\binom{n+k}n\cdot\frac1{n+k}=0.\end{equation*} 

To show this, define $a_k$ and $b_k$ as \[a_k=(-1)^k\binom nk,\quad b_k=\binom{n+k}n\cdot\frac1{n+k}.\] Then we have $S_n=\sum_{k=0}^na_{n-k}b_k$. It suffices to show that $S_n=0$ whenever $n\ge1$. To do this, we will use the generating functions of the sequences $\{a_k\}$ and $\{b_k\}$. In particular, begin by defining $A(x)$ and $B(x)$ as \begin{align*}A(x)&\coloneqq\sum_{k=0}^na_kx^k=\sum_{k=0}^n(-1)^k\binom nkx^k=(1-x)^n\\B(x)&\coloneqq\sum_{k=0}^\infty b_kx^k=\sum_{k=0}^\infty\binom{n+k}n\cdot\frac1{n+k}x^k.\end{align*} Then it is clear that \[\frac d{dx}\Big(x^nB(x)\Big)=x^{n-1}\sum_{k=0}^\infty\binom{n+k}nx^k=(1-x)^{-n-1}-1,\] where the second equality holds for all $|x|<1$. Taking the integral and simplifying, we find that \[x^nB(x)=\frac{x^n}{n(1-x)^n}.\] Note that there is no constant to worry about because at $x=0$, both sides are zero. 

In short, we have now that \begin{align*}A(x)&=\sum_{k=0}^na_kx^k=(1-x)^n\\B(x)&=\sum_{k=0}^\infty b_kx^k=\frac1{n(1-x)^n}.\end{align*} But we also know that the coefficient of $x^n$ in $A(x)B(x)=\frac1n$ is $\sum_{k=1}^na_{n-k}b_k=S_n$. From this we conclude that $S_n=0$, as desired. 
\end{proof} 

\begin{corollary} \label{cor: nC_n recurrence}
For every postive integer $n$, we have \begin{equation*}
nC_n=\sum_{k=1}^n(-1)^{k+1}\left[\binom{2n-k}k+\binom{2n-k-1}{k-1}\right]C_{n-k}(n-k).\end{equation*}
\end{corollary} 

\begin{proof}
This follows by combining \Cref{lem: C_n recurrence} and \Cref{lem: nC_n recurrence}. In particular, for every positive integer $n$, we know that \[nC_n=C_n(n+1)-C_n=\sum_{k=1}^n(-1)^{k+1}a_{k,n}C_{n-k}(n-k),\] where $a_{k,n}=\binom{2n-k}k+\binom{2n-k-1}{k-1}$. 
\end{proof}

Combining \Cref{lem: C_n recurrence,cor: nC_n recurrence}, along with previously known results for small cases detailed in \Cref{sec: known results}, we are able to show that $C_n$ satisfies the recurrence for $N(n,-p,q)$ in \Cref{lem: N(npq) recurrence}. This thus proves the main theorem. 

\begin{proof}[Proof of \Cref{thm: main theorem}]
We know by \Cref{lem: N(npq) recurrence} that \[N(n,-p,q)=\sum_{k=1}^n(-1)^{k+1}\left[\binom{2n-k}k+\binom{2n-k-1}{k-1}\right]N(n-k,p,q).\] However, we also know that for any $r$ and $s$, the numbers $C_n((r-s)n+s)$ also satisfy this recurrence. In particular, if we once again let $a_{k,n}=\binom{2n-k}k+\binom{2n-k-1}{k-1}$, then we know that \begin{align*}C_n((r-s)n+s)&=(r-s)nC_n+sC_n\\&=\sum_{k=1}^n(-1)^{k+1}a_{k,n}C_{n-k}((r-s)(n-k)+s).\end{align*}

It therefore suffices to prove that the two sequences $\{N(n,-p,q)\}$ and $\{C_n((r-s)n+s)\}$ agree on $n=0,1$, for specific $r$ and $s$ dependent only on $p$ and $q$. But, writing $-p/q=[r_0,r_1,\dots,r_k]$, we know that we have \[N(1,p,q)=|(r_0+1)\dots(r_{k-1}+1)r_k|=r=C_1((r-s)\cdot1+s).\] 

Moreover, we know that, with $p'$ and $q'$ defined by the attachment of a bypass to $(1,-p,q)$, we must have \[-\frac{p'}{q'}=[r_0,r_1,\dots,r_{k-1},r_k+1].\] Recall that we are using $N(0,p,q)$ as shorthand to mean $N(1,p',q')$. Thus we find that \[N(0,p,q)=N(1,p',q')=s=C_0((r-s)\cdot0+s).\]

Our main theorem then follows, giving us \[N(n,-p,q)=C_n((r-s)n+s),\] as desired. 
\end{proof}

\bibliography{Index}

\end{document}